\documentclass[11pt]{amsart}
\usepackage{amsmath,amsthm,amssymb,tikz,mathtools,amsfonts,graphicx}
\usepackage[applemac]{inputenc}
\usepackage{amsfonts,txfonts}

\theoremstyle{plain}
\newtheorem{theorem}{Theorem}
\newtheorem{lemma}[theorem]{Lemma}
\newtheorem{claim}{Claim}
\newtheorem{prop}[theorem]{Proposition}
\newtheorem{corollary}[theorem]{Corollary}
\theoremstyle{remark}
\newtheorem{remark}[theorem]{Remark}
\newtheorem*{remark*}{Remark}

\numberwithin{equation}{section}
\theoremstyle{definition}

\renewcommand{\d}{\mathop{}\mathopen{}\mathrm{d}}

\newcommand{\ve}{\varepsilon}

\newcommand{\deb}{\rightharpoonup}

\newcommand{\R}{\mathbb{R}}

\newcommand{\W}{\mathcal{W}}

\newcommand\gd{\mathrm{geo}d}

\DeclareMathOperator{\Lip}{Lip}

\begin{document}
\author{Antonin Monteil, Filippo Santambrogio}
\title{Metric methods for heteroclinic connections\\ in infinite dimensional spaces}
\address{{\bf A.M.} Institut de Recherche en Math\'ematiques et Physique,
Universit\'e Catholique de Louvain,
Chemin du Cyclotron 2 Bte L7.01.01, 
1348 Louvain-la-Neuve, Belgium, 
{\tt antonin.monteil@uclouvain.be}\\
{\bf F.S.} Laboratoire de Math\'ematiques d'Orsay, 
Univ. Paris-Sud, CNRS, Universit\'e Paris-Saclay, 
91405 Orsay Cedex, France, 
{\tt  filippo.santambrogio@math.u-psud.fr}}

\begin{abstract}We consider the minimal action problem $\min\int_\R \frac{1}{2}|\dot{\gamma}|^2+W(\gamma)\d t $ among curves lying in a non-locally-compact  metric space and connecting two given zeros of $W\geq 0$. For this problem, the optimal curves are usually called heteroclinic connections. 
We reduce it, following a standard method, to a geodesic problem of the form $\min\int_0^1 K(\gamma)|\dot{\gamma}|\d t$ with $K=\sqrt{2W}$. We then prove existence of curves minimizing this new action under some suitable compactness assumptions on $K$, which are minimal. The method allows to solve some PDE problems in unbounded domains, in particular in two variables $x,y$, when $y=t$ and when the metric space is an $L^2$ space in the first variable $x$, and the potential $W$ includes a Dirichlet energy in the same variable. We then apply this technique to the problem of connecting, in a functional space, two different heteroclinic connections between two points of the Euclidean space, as it was previously studied by Alama-Bronsard-Gui and by Schatzman more than fifteen years ago. With a very different technique, we are able to recover the same results, and to weaken some assumptions.
\end{abstract}

\maketitle



\section{Introduction}

The minimal action problem, directly coming from Newtonian mechanics, is, in its most classical form, a variational problem where an energy of the form
\begin{equation}\label{het1}
\left(\gamma:I\to\R^d\right)\mapsto \mathfrak{E}_{W}(\gamma):= \int_I\left( \frac{1}{2}|\dot{\gamma}|^2(t)+W(\gamma(t))\right)\d t 
\end{equation}
is minimized among curves connecting two given points. In a smooth setting the corresponding Euler-Lagrange equation is $\gamma''=\nabla W(\gamma)$, which is one of the simplest and most studied second-order differential equation. When the time interval $I$ is the whole line $\R$ and $W\geq 0$, solutions of this equation (or of the minimization problem) connecting at $\pm\infty$ two points $a^\pm$ where $W=0$ (which is a necessary condition for the action to be finite) are called heteroclinic connections. 

The existence of a heteroclinic connection is a very delicate problem, because of the lack of compactness of the set $H^1(\R)$ and of the invariance by translations of the action to be minimized. We cite \cite{SteRocky,Ali Fus,AntSmy,Sourdis,Ali Fus2015} among the many papers dealing with this and related question. In a previous paper \cite{nous} we analyzed the same question via a purely metric method. 

The idea behind the method was classical: reduce the problem to a geodesic problem for a weighted metric with a cost given by $K(x):=\sqrt{2 W(x)}$, i.e., instead of minimizing \eqref{het1}, solving
$$
\min\quad \mathfrak{L}_{K}(\gamma):=\int_0^1 K(\gamma(t))\, |\gamma'(t)|\d t.
$$
The connection between the two problems comes from the Young inequality, which gives
$$\frac12 |\gamma'|^2+W(\gamma)\geq \sqrt{2W(\gamma)}\, |\gamma'|,\quad\mbox{ with equality if and only if }\sqrt{2W(\gamma)}=|\gamma'|.$$
The weighted length $ \mathfrak{L}_{K}$ is invariant under parametrization, which is why one can reduce the problem to the interval $[0,1]$. Then, if one is able to find a minimizer for $ \mathfrak{L}_{K}$, it is enough to choose a suitable reparametrization of it on $\R$ which satisfies $\sqrt{2W(\gamma)}=|\gamma'|$, and this will be a minimizer for $\mathfrak{E}_{W}$.

If in this way we get rid of the difficulty given by the non-compactness of $\R$, we face now a new difficulty, the fact that obtaining Sobolev bounds for a minimizing sequence requires lower bounds on $K$, while $K=\sqrt{2 W}$ exactly vanishes at the two wells $a^\pm$. 

However, in \cite{nous}, we managed to overcome this difficulty in the case where the curves $\gamma$ lie in the Euclidean space $\R^n$ and the weight $K\geq 0$ is continuous. To do so, we studied the space $(\R^n,d_K)$, i.e. the same Euclidean space endowed with the geodesic distance induced by $K$. This distance is the one defined by $d_K(x,y):=\min   \mathfrak{L}_{K}(\gamma)$, the minimum being taken among curves connecting $x$ to $y$. We proved that such a space is a {\it proper space}, i.e. bounded sets are pre-compact, which guarantees the existence of geodesic curves.

The main point which motivated \cite{nous}, besides recovering classical results on the heteroclinic problem in $\R^n$, was the fact that all the study was done in a more general metric space (with $|\gamma'|$ which is defined as the metric derivative, see \cite{AmbTil}), which allowed for many generalizations. The most interesting one is the following. Consider a higher-dimensional problem, as for instance
\begin{equation}\label{genOmega}
\min \int_{\R\times \Omega}\left( \frac{1}{2}|\nabla u|^2(x)+W(u(x))\right)\d x,
\end{equation}
where $x=(x_1,x')$, and boundary data are fixed as $x_1\to\pm\infty$. This can be interpreted in our framework using $x_1$ as $t$ and $X$ to be $L^2(\Omega)$, with an effective potential of the form ${v}\in L^2(\Omega)\mapsto \mathcal W({v}):=\int_\Omega \frac{1}{2}|\nabla_{x'} {v}|(x')^2+W({v}(x'))\d x'$. This obviously raises extra difficulties due to the lack of compactness in infinite dimensions, for which but the key point to overcome it will be the fact that the sublevels of $\mathcal W$ have extra compactness properties.

The goal of the present paper is exactly to develop this project: studying heteroclinc connections in functional spaces as a particular case of  a metric setting, finding sharp conditions for the existence of weighted geodesics. The title of the paper underlines the fact that we mainly deal with metric spaces which are infinite-dimensional (actually, the main feature is that they are no more locally compact, differently from the framework of \cite{nous}).

The content of the paper is the following. After recalling in Section \ref{sectionPrelim} the main notions concerning curves and geodesics in metric spaces, in Section \ref{mainthm} we consider the abstract problem of minimizing a weighted length in a metric space, with a lower-semicontinuous weight $K$ which can possibly vanish. Of course, it is necessary that it does not vanish too much, and we require the vanishing set $\Sigma=\{K=0\}$ to be finite; some conditions on the behavior of $K$ ``at infinity'' are also typically required. While many authors require a lower bound of the form $\liminf_{d(x,\Sigma)\to\infty}K(x)>0$, in \cite{nous} we required a non-integrability condition $K(x)\geq k(d(x,\Sigma))$ with $\int_0^\infty k(s)\d s=+\infty$. This allowed to reduce the problem to bounded balls but, in the case of a non-proper space $X$, this is in general not enough, and we require a more general (and abstract) condition. However, in the Appendix A we discuss this lower bound assumption by means of a counter-example. In Section \ref{mainthm}, several equivalent definitions of the weighted length $\mathfrak{L}_{K}$ are introduced, as we need to prove its semicontinuity; this is one of the main difficulties, together with the proof of a suitable equicontinuity of minimizing sequences.

In Section \ref{sectionHeteroclinic} we use the existence results of Section \ref{mainthm} (which deal with weighted geodesics, i.e. minimizers of $\mathfrak{L}_{K}$) to provide existence of heteroclinic connections (minimizers of $\mathfrak{E}_{W}$) by detailing the reparametrization procedure (which is delicate because $K$ is not supposed to be continuous). We also present a first easy application of these results to problems of the form \eqref{genOmega}, in the case of bounded $\Omega$. 

Then come sections \ref{sectionDouble}, \ref{sectionAlama} and \ref{sectionSchatzman}, which take most of the paper. These sections are devoted to a very natural and very interesting problem: given a double-well potential $W$ on $\R^n$, consider the heteroclinic connection problem between its two wells, and suppose that it admits two distinct solutions. These two solutions are curves, belonging to a functional space included in $L^2_{loc}(\R,\R^n)$; in this space we want to connect these two curves. Essentially this amounts to finding a solution of 
\begin{equation}
\label{localReactionDiffusion}
\begin{cases}
-\Delta u+\nabla W(u)=0&\text{over $\R^2$ \big(or we can require $u$ to be a local minimizer of $\int \frac 12 |\nabla u|^2+W(u)$\big)};\\
u(x_1,x_2)\to a^-&\text{when }x_1\to -\infty,\text{ uniformly w.r.t. }x_2;\\
u(x_1,x_2)\to a^+&\text{when }x_1\to +\infty,\text{ uniformly w.r.t. }x_2;\\
u(x_1,x_2)\to z^-(x_1)&\text{when }x_2\to -\infty,\text{ uniformly w.r.t. }x_1;\\
u(x_1,x_2)\to z^+(x_1)&\text{when }x_2\to +\infty,\text{ uniformly w.r.t. }x_1;
\end{cases}
\end{equation}
where $z^\pm$ are the two curves that we need to connect and $a^\pm=z^\pm(\pm\infty)$ are the two wells of $W$ ; the existence of a solution has been established with various assumptions in \cite{Alama:1997,Alessio2012} for symmetric solutions and in \cite{Schatzman:2002} (see also the recent paper \cite{fusco2017layered} for an alternate proof) in the non-symmetric case. The preceding system arises, for instance, in the study of the local behavior of solutions to the reaction-diffusion system,
\[
\partial_t u(t,x)-\varepsilon^2\Delta u(t,x)+\nabla W(u(t,x))=0,\quad x\in\Omega\subset\R^2,\, t>0,
\]
in the asymptotic regime $\varepsilon\to 0$. As $\varepsilon$ tends to $0$, solutions converge almost everywhere to minima of $W$, thus revealing sharp interfaces separating distinct phases. As it holds for the scalar Allen-Cahn equation, one might expect that, near a point of the interface between the two phases $a^-$, $a^+$ and in a first order approximation, solutions only depend on the orthogonal (to the interface) variable, and that the dependance on this variable corresponds to a {\it stationary wave} (or heteroclinic connection), i.e. a 1D solution of the first three equations of \eqref{localReactionDiffusion}. However, for vector-valued equations, several distinct heteroclinic connections between $a^-$ and $a^+$ might exist and a solution to the full system \eqref{localReactionDiffusion} (with two distinct heteroclinic solutions $z^-$ and $z^+$) is in particular not 1D. This makes a significative difference with the scalar framework, since the De Giorgi conjecture \cite{de1979convergence} (see \cite{Ghoussoub:1998,AmbrosioCabre:2000,ghoussoub2003giorgi,Savin:2009} for the proofs in dimensions $d=2$, $d=3$, $d=4,5$ and $d=6,7,8$ respectively) claims that the equation $-\Delta u+\nabla W(u)=0$ has a unique (up to translation) non-trivial solution $u:\R^d\to\R$, monotone in $x_1$-direction, and it thus corresponds to the unique stationary wave, i.e. $u$ only depends on $x_1$. This brings to light a more complex local behavior of vector-valued reaction-diffusion systems near a point of the interface, where solutions can depend on the tangential variable (corresponding to $x_2$ in the rescaled system \eqref{localReactionDiffusion}), in such a way that it connects two distinct stationary waves.

A main difficulty in the study of the system \eqref{localReactionDiffusion} is the fact that minimizers of $\mathfrak{E}_{W}$ on $I=\R$ are always defined up to translations, so that they are not really finite in number. In \cite{Alama:1997} a particular case is considered: the case where $W$ has some symmetries and we look for symmetric connections $z^\pm$. This rules out the translation invariance and allows to study a case where it is reasonable to assume that the number of heteroclinic connections between $a^-$ and $a^+$ is two. By using a metric space $X$ of symmetric curves defined on $\R$ endowed with the $L^2$ distance, we recover via our metric approach, in Section \ref{sectionAlama}, the existence result of \cite{Alama:1997}. To prepare for this result, in Section \ref{sectionPrelim}, we present some preliminaries, including an original and very useful Lemma \ref{projection} which claims that the energy $\mathfrak{E}_{W}$ of a curve $z$ decreases if we project $z$ onto the set of curves satisfying an inequality of the form $|z(x_1)-a^\pm|\leq E(x_1)$ for $|x_1|\geq M$, the shape of the profile $E$ being chosen in a suitable way according to the degeneracy of the potential $W$ around $a^\pm$; we call these profiles, which can decrease algebraically or exponentially to $0$ as $x_1\to \infty$, {\it funnels} because of their shape. Then, in Section \ref{sectionSchatzman}, we consider a much more difficult case, the case where the symmetry condition is removed, and the heteroclinic connections between $a^-$ and $a^+$ are considered up to translations. In this case $X$ is a quotient of a linear space, and the difficulty is to provide a solution $u$ which admits a true limit (and not only up to translations) as $x_1\to\pm\infty$. This case was succesfully studied in \cite{Schatzman:2002}, and we recover the very same result with our technique. In both cases (that of \cite{Alama:1997} and that of \cite{Schatzman:2002}), the uniform convergence as $x_1\to\pm\infty$ is obtained by using the funnel lemma, while the uniform convergence as $x_2\to\pm\infty$ requires an extra argument, as our technique only provides $L^2$ convergence. This improvement of the convergence (and also an improvement of the regularity of the solutions, which will finally belong to $C^{2,\alpha}$, or to better spaces according to the regularity of the potential $W$) is obtained by means of a general fact, presented in Appendix B, about minimal action curves for $\lambda$-convex functionals in Hilbert spaces. The result presented in the appendix allows indeed to obtain a bound on $\mathcal W(u(x_1,\cdot))$ independent of $x_1$, which provides $ H^1\cap L^\infty$ bounds. This  allows both to transform the $L^2$ convergence into $L^\infty$ and to bound the right-hand side of the elliptic PDE $\Delta u=\nabla W(u)$.

It is interesting to compare the assumptions on the potential $W$ which are used in \cite{Alama:1997,Schatzman:2002} and in other papers on heteroclinic connections to those that we use in the present paper. We already pointed out that assumptions of the form $\liminf_{|x|\to\infty} W(x)>0$ can be easily replaced with $W(x)\geq k^2(|x|)$, $\int_0^\infty k=+\infty$, which we do. Another typical assumption in the literature is $\nabla W(x)\cdot x\geq 0$ for $|x|\geq R$ (or, more generally, the fact that $W(x)\geq W(Rx/|x|)$ for $|x|\geq R$): this guarantees that projecting onto a ball containing the wells $a^\pm$ decreases the energy, and allows to obtain boundedness of the competitors. Yet, this simplifying assumption is not compatible with a potential $W$ tending to $0$ at infinity, and we preferred not to use it. Our $L^\infty$ bounds are obtained a posteriori using the above lower bound $W(x)\geq k^2(|x|)$, together with a $\lambda$-convexity argument presented in Appendix \ref{AppB}. In order to apply this argument, and also to apply a regularization procedure in another part of the proof,  we need a global on the Hessian $\nabla^2 W$, that we suppose bounded from below (i.e. $W$ is supposed to be $\lambda$-convex for some negative $\lambda$).

We can summarize the paper by saying that its main contributions are the following:
\begin{itemize}
\item a general and abstract existence result for geodesics in weighted metric spaces, under suitable compactness conditions (Theorem \ref{exist_Kgeod});
\item the application of the above result to the existence of heteroclinic connections in general metric spaces (Theorem \ref{heteroclinic});
\item the application of these techniques to the existence result originally presented in  \cite{Alama:1997}, under slightly weaker assumptions (Theorem \ref{ABGthm}, where more degenerate potentials $W$, compared to \cite{Alama:1997}, are admitted);
\item the application of the same techniques to the existence result originally presented in  \cite{Schatzman:2002}, essentially under the same assumptions (Theorem \ref{Schatzman_thm}).
\end{itemize}

\noindent {\bf Acknowledgments.} The authors warmly acknowledge Nick Alikakos for pointing out to them the problems studied by Alama-Bronsard-Gui and Schatzmann, for his constant interest into this project, and for the warm hospitality in Athens.

\section{Minimal length problem in metric spaces\label{sectionPrelim}}

Let $(X,d)$ be a metric space: $X$ is a set and $d:X\times X\to [0,+\infty]$ is a metric, i.e. $d$ is symmetric, subadditive and vanishes on the diagonal, and only on the diagonal, of $X\times X$. Note that in our definition of a distance, we do not assume $d$ to be finite. This is more convenient for our purpose as we will consider a distance $d_K$ which needs not be finite everywhere. In the sequel, we will use the notation $B_d(x,r)$ (resp. $\overline{B}_d(x,r)$) for the open (resp. closed) ball centered at $x$ of radius $r\geq 0$:
\[
B_d(x,r)=\{y\in X\;:\; d(y,x)< r\}\quad \text{and}\quad \overline{B}_d(x,r)=\{y\in X\;:\; d(y,x)\leq r\}.
\]

\paragraph{\bf Curves in $(X,d)$} A \emph{curve} is a continuous map $\gamma:I\to X$, where $I\subset \R$ is a non-empty interval. The set of all curves $\gamma:I\to X$, denoted by $\mathcal{C}(I,X)$, is endowed with the topology of uniform convergence on compact subsets of $I$. 
We denote the set of Lipschitz maps (resp. locally Lipschitz maps) from $I$ to $X$ by $\Lip(I,X)$ (resp. $\Lip_{loc}(I,X)$). We also need to introduce the set of \emph{piecewise locally Lipschitz maps}:
\[
\Lip_{ploc}(I,X):=\big\{\gamma\in \mathcal{C}(I,X)\;:\;\exists t_0=\inf I<t_1<\dots <t_n=\sup I,\, \forall i,\,\gamma\in\Lip_{loc}( (t_i,t_{i+1}), X)\big\}.
\]

We will also need to consider less regular maps, namely absolutely continuous curves. We denote the set of absolutely continuous maps (resp. locally absolutely continuous maps) from $I$ to $X$ by $AC(I,X)$ (resp. $AC_{loc}(I,X)$). We remind that $\gamma\in AC(I,X)$ if and only if for all $\epsilon>0$, there exists $\delta>0$ such that for all sequences $t_0<t_1<\dots<t_n$ in $I$, one has
\[
\sum_{i=0}^{n-1} |t_{i+1}-t_i|<\delta\quad\Longrightarrow\quad\sum_{i=0}^{n-1}d(\gamma(t_{i+1}),\gamma(t_i))<\varepsilon.
\]
As before, we also need to introduce the set of \emph{piecewise locally absolutely continuous maps}:
\[
AC_{ploc}(I,X):=\big\{\gamma\in\mathcal{C}(I,X)\;:\;\exists t_0=\inf I<t_1<\dots <t_n=\sup I,\, \forall i,\,\gamma\in AC_{loc}((t_i,t_{i+1}),X)\big\}.
\]

\paragraph{\bf Length of a curve}
For every map $\gamma:I\to X$, we define the length of $\gamma$ by the usual formula
$$L_d(\gamma):=\sup \sum_{i=0}^{N-1} d(\gamma(t_i),\gamma(t_{i+1}))\in[0,+\infty],$$
where the supremum is taken over all $N\geq 1$ and all sequences $t_0\leq \dots\leq t_N$ in $I$. A map $\gamma$ is said to be \emph{rectifiable} if it is a curve (i.e. $\gamma$ is continuous) and $L(\gamma)<\infty$.

\medskip
\paragraph{\bf Length of absolutely continuous curves} For piecewise locally absolutely continuous maps we have the following representation formula for the length:
\begin{prop}\label{length_lip}
Given $\gamma\in AC_{ploc}(I,X)$, the following quantity,
\[
|\dot{\gamma}|(t)=\lim\limits_{s\to t}\frac{d(\gamma(t),\gamma(s))}{|t-s|},
\]
is well defined for a.e. $t\in I$ and $|\dot{\gamma}|(\cdot)$, called \emph{metric derivative} of $\gamma$, is measurable. Moreover, one has
\[
L_d(\gamma)=\int_I |\dot{\gamma}|(t)\d t .
\]
\end{prop}
We refer for instance to \cite{AmbTil} for the notion of metric derivative and for many other notions on the analysis of metric spaces.

\medskip
\paragraph{\bf Parametrization}
If $\gamma:I\to X$ is any map, and $\varphi:I'\to I$ is a non-decreasing surjective (and thus continuous) mapping, called \emph{parametrization}, then the curve $\sigma=\gamma\circ\varphi :I'\to X$ satisfies $L_d(\sigma)=L_d(\gamma)$. The map $\gamma$ is said to have \emph{constant speed} if there exists $\lambda\in\R_+$ such that for all $t,t'\in I$ such that $t<t'$, $L_d(\gamma_{|(t,t')})=\lambda |t-t'|$. Then $\lambda$ is said to be the \emph{speed} of the curve $\gamma$. Note that $\gamma$ has constant speed $\lambda$ if and only if $\gamma$ is Lipschitz and $|\dot{\gamma}(t)|=\lambda$ a.e. The curve $\gamma$ is said to be {\it parametrized by arc length} if $\lambda=1$. 

Assume that a curve $\gamma$ satisfies $L_d(\gamma_{|J})<\infty$ for all compact subset $J\subset I$. Then there exists a paramatrization of $\gamma$ by arc length, obtained as follows. Let us fix $t_0\in I$ and $\varphi(t):=\pm L_d(\gamma_{|(t_0,t)})$ for $t\in I$ s.t. $\pm (t-t_0)\geq 0$. Then $\varphi$ is continuous, non-decreasing and the curve
\[
\sigma: \varphi(I) \to X\ ,\quad \sigma(\varphi(t))=\gamma(t)
\]
is well defined, continuous and parametrized by arc length. Indeed, for $t,t'\in I$ such that $t\leq t'$, we have $\varphi(t')-\varphi(t)=L_d(\gamma_{|(t,t')})=L_d(\sigma_{|(\varphi(t),\varphi(t'))})$. 

Up to renormalization, it is always possible to consider curves defined on $I=[0,1]$.

\medskip
\paragraph{\bf Minimal length problem} We define the intrinsic pseudo-metric $\gd$ (called {\it geodesic distance}) by minimizing the length of all curves $\gamma$ connecting two points $x^\pm\in X$:
\begin{equation}
\label{intrinsic}\gd(x^-,x^+):=\inf\{L_d(\gamma)\;:\; \gamma:x^-\mapsto x^+\}\in [0,+\infty],
\end{equation}
where the notation $\gamma:x^-\mapsto x^+$ means that $\gamma$ is a \emph{path} from $x^-$ to $x^+$: there exists an interval $I\subset\R$ s.t. $\gamma\in\mathcal{C}(I,X)$ with $\gamma(a^\pm)=x^\pm$, where $a^-=\inf I$ and $a^+=\sup I$. Here, if $a^+$ or $a^-$ does not belong to $I$, the notation $\gamma(a^\pm)=x^\pm$ means $\lim_{t\in I\to a^\pm}\gamma (t)=x^\pm$.

When $(X,d)$ is a normed vector space, $\gd=d$ and the infimum value in \eqref{intrinsic} is achieved by the segment $[x^-,x^+]$. In general, a metric space such that $\gd=d$ is called a $\emph{length space}$.

The minimal length problem consists in finding a curve $\gamma:x^-\mapsto x^+$ s.t. $L_d(\gamma)=\gd(x^-,x^+)$. When $X$ is proper, the existence of such a curve, called \emph{minimizing geodesic}, is given by the classical theorem (see \cite{AmbTil}, for instance):
\begin{theorem}\label{exist_geod}
Assume that $(X,d)$ is proper, i.e. every bounded closed subset of $(X,d)$ is compact. Then, for any two points $x^\pm$ s.t. $\gd(x^+,x^-)<+\infty$, there exists a minimizing geodesic joining $x^-$ and $x^+$.
\end{theorem}

\section{Minimal length problem in weighted metric spaces\label{mainthm}}

Let $(X,d)$ be a metric space and $K: X\to [0,+\infty]$ be a nonnegative measurable function, called \emph{weight function}. Our aim is to investigate the existence of a curve $\gamma$ minimizing the $K$-length, defined by
\[
\mathfrak{L}_{K}(\gamma):=\int_I K(\gamma(t))\, |\dot{\gamma}|(t)\d t,\quad \gamma\in AC_{ploc}(I,X).
\]
There is an ambiguity in the definition when $K(\gamma(t))=+\infty$ and $|\dot{\gamma}|(t)=0$. We will use the convention $+\infty\times 0=+\infty$. Minimizing $\mathfrak{L}_{K}$ along curves $\gamma:x^-\mapsto x^+$ allows to define the \textit{$K$-distance} between given points $x^-,x^+\in X$:
\begin{equation}\label{Kdistance}
d_K(x^-,x^+):=\inf\left\{\mathfrak{L}_{K}(\gamma)\;:\; \gamma\in AC_{ploc}(I,X)\text{ s.t. }\gamma:x^-\mapsto x^+\right\}\in [0,+\infty].
\end{equation}
In order to prove existence of a minimizing curve, we will need the following conditions, that are assumed to be satisfied in the whole paper, unless otherwise specified. These assumptions concern the space $(X,d)$, the cost $K$, and the points $x^\pm$ to be connected:
\begin{description}
	\item[(H1)]
	$(X,d)$ is complete and is a length space;
	\item[(H2)]
	$K:X\to [0,+\infty]$ is lower semicontinuous and $\Sigma:=\{K=0\}$ is finite;
	\item[(H3)]
there exists a subset $F$ of $X$ such that all the intersections $F \cap \{K\leq \ell\}$ for $\ell<+\infty$ are compact sets, and
\[
d_K(x^-,x^+)=\inf\left\{\mathfrak{L}_{K}(\gamma)\;:\; \gamma\in AC_{ploc}([0,1],X)\text{ s.t. }\gamma:x^-\mapsto x^+\text{ and }\mathrm{Im}(\gamma)\subset F\right\}.
\]
\end{description}
Assumption {\bf (H1)} is satisfied in particular by any Banach space. Assumption \textbf{(H3)} is for instance satisfied when the two following conditions are fulfilled:
\begin{description}
\item[(H3a)] for all $x\in X$, $K(x)\geq k(d(x,\Sigma))$ for some function $k\in C^0(\R^+,\R^+)$ with $\int_0^\infty k(t)\d t=+\infty$;
\item[(H3b)] for all $R,\ell\in (0,+\infty)$ and $x_0\in X$, $\overline{B(x_0,R)}\cap \{K\leq\ell\}$ is compact in $(X,d)$. 
\end{description}
Indeed, the first of these two facts implies \cite[Proof of Proposition 2, Step 3]{nous} that curves with bounded $\mathfrak{L}_{K}$-length stay in a bounded set, say a ball $\overline{B(x_0,R)}$, and allows to use $F=\overline{B(x_0,R)}$ in Assumption \textbf{(H3)}. Moreover, if $X$ is a proper space, then condition {\bf (H3b)} is automatically satisfied, since $\{K\leq\ell\}$ is closed for every $\ell$ (because $K$ is l.s.c.) and closed balls in $X$ are compact. This explains the importance of condition {\bf (H3a)} when we face proper spaces. The case where $(X,d)$ is proper and $K:X\to [0,+\infty)$ is continuous was treated in \cite{nous}. 
In Appendix A we will see a counterexample when $X$ is proper but condition {\bf (H3a)} fails, showing that this is a natural assumption.

On the contrary, the present paper will be mainly concerned with infinite dimensional settings, where $X$ is in general not proper. In some cases intersecting balls with sublevel sets $\{K\leq\ell\}$ is enough to obtain compactness. In particular, this is the case when $X=L^2(\Omega)$ with $\Omega\subset\R^N$ open and bounded, and we define $K^2(u)=\|\nabla u\|^2_{L^2(\Omega)}+\int_\Omega f(x,u(x))dx$ if $u\in H^1(\Omega)$ and $+\infty$ otherwise, where $f:\Omega\times\R\to\R$ is a continuous and bounded function (see Section \ref{sectionHeteroclinic}).

In other cases (which will be of interest in Sections 5, \ref{sectionAlama} and \ref{sectionSchatzman}), intersecting balls with sublevel sets $\{K\leq\ell\}$ will not be enough to obtain compactness, which explains why we introduced the more general assumption  {\bf (H3)}.\bigskip

We are going to prove that the \emph{K-distance} $d_K$ is a metric on $X$ (possibly infinite), and that $\mathfrak{L}_{K}=L_{d_K}$ (see Proposition \ref{Klength} below). The main result of this section is that $(X,d_K)$ is a geodesic space. More precisely, one has the following theorem:
\begin{theorem}\label{exist_Kgeod}
Given $x^\pm\in X$ with $d_K(x^+,x^-)<\infty$, there exists a curve $\gamma\in \Lip_{ploc}(I,X)$ s.t. $\mathfrak{L}_{K}(\gamma)=d_K(x^+,x^-)$ and $\gamma:x^+\mapsto x^-$.
\end{theorem}
Before attacking the proof of Theorem \ref{exist_Kgeod}, we need to prove a topological proposition which contains the corner stone of the proof. We will use the set $F$ evoked by Assumption ${\bf (H3)}$.
\begin{prop}\label{Klength}
The quantity $d_K$ defines a metric on $X$ and every $d_K$-bounded set contained in $F$ is precompact in $(X,d)$,
\end{prop}
\begin{proof} 
It is clear that $d_K:X\times X\to [0,+\infty]$ is nonnegative and symmetric. Moreover $d_K$ satisfies the triangle inequality since the K-length is additive (the K-length of a curve obtained by concatenation is the sum of the K-lengths of each curve). Obviously, one has $d_K(x,x)=0$ whatever $x\in X$ since $d_K(x,x)\leq \mathfrak{L}_{K}(\gamma)=0$ if $\gamma:\{0\}\to X$ is the constant map given by $\gamma(0)=x$ (note that a constant map $\gamma:I\to X$ on a non-trivial interval $I$ needs not satisfy $\mathfrak{L}_{K}(\gamma)=0$ since $\mathfrak{L}_{K}(\gamma)=+\infty$ if $\gamma\equiv x$, with our convention $+\infty\times 0=+\infty$). The fact that $d_K(x,y)=0$ implies $x=y$ follows from Assumption {\bf (H2)}. Indeed, let $x\neq y$ be two distinct points in $X$ and let $\gamma:I\to X$ be a piecewise absolutely continuous curve joining $x$ to $y$. Then, by continuity of $t\mapsto d(x,\gamma(t))$, there exist $t_1< t_2$ s.t. $d(\gamma(t_1),x)=\varepsilon$, $d(\gamma(t_2),x)=2\varepsilon$ and $\varepsilon\leq d(\gamma(t),x)\leq 2\varepsilon$ for $t_1\leq t\leq t_2$. This implies that $\mathfrak{L}_{K}(\gamma)\geq \varepsilon\, \inf_C K$, where $C=\gamma([t_1,t_2])$. Yet, for $\varepsilon$ small enough, $C$ does not intersect the set $\{K=0\}$ so that $\inf_C K>0$ as $C$ is compact and $K$ is lower semicontinuous. In particular, $d_K(x,y)>0$.
\medskip

%
%

We now prove that $d_K$-bounded sets contained in $F$ are $d$-precompact. This means proving that for every ball $B:=\overline{B}_{d_K}(x_0,r)$, with $x_0\in X$ and $r>0$, $F\cap B$ is precompact in $(X,d)$. Thanks to Assumption {\bf (H3)}, the set $F_\ell:=\{x_0\}\cup\{x\in F\;:\;K(x)\leq \ell\}$ is compact for all $\ell>0$. Thus it is enough to prove that for all $\varepsilon>0$, there exists $\ell>0$ s.t. $F\cap B\subset (F_\ell)^\varepsilon:=\{x\in X\;:\; d(x,F_\ell)\leq \varepsilon\}$. This would allow to cover $F\cap B$ with a finite $\varepsilon$-net, which shows precompactness. We just need to prove that, given $\varepsilon>0$, any point $x\in F\cap B$ lies within $\varepsilon$-distance to a point $y$ s.t. $K(y)\leq \ell$, where $\ell$ only depends on $\varepsilon$, $x_0$ and $r$. As $x\in B$, there exists $\gamma\in AC_{ploc}([0,1],X)$ s.t. $\gamma: x_0\mapsto x$ and $\mathfrak{L}_{K}(\gamma)\leq 2r$. If $L_d(\gamma)\leq \varepsilon$, then $x\in \overline{B}_d(x_0,\varepsilon)\subset  (F_\ell)^\varepsilon$, and if $K(x)=0$, then $x\in F_0$. Otherwise, by continuity of $t\mapsto\varphi(t):=L_d(\gamma_{|[t,1]})$ and since $\varphi(t)<\infty$ for $t$ close to $1$ (because $K$ is bounded from below on $Y:=\gamma([1-\eta,1])$ for some $\eta>0$, and $L_d(\gamma_{|[1-\eta,1]})\inf_Y K\leq \mathfrak{L}_{K}(\gamma_{|[1-\eta,1]})$), there exists $t_0\in [0,1]$ s.t. $L_d(\gamma_{|[t_0,1]})=\varepsilon$. Now, as $\mathfrak{L}_{K}(\gamma_{|[t_0,1]})=\int_{t_0}^1 K(\gamma)|\dot{\gamma}| \leq 2r$, there exists $t\in [t_0,1]$ s.t. $K(\gamma(t))\leq \frac{2r}{\varepsilon}$. Thus $y:=\gamma(t)\in F_\ell$ with $\ell:=\frac{2r}{\varepsilon}$ and $d(y,x)\leq\varepsilon$ as required.
%
%
\end{proof}

We also need the following metric identities:
\begin{prop}\label{metric}
For all curves $\gamma\in AC_{ploc}(I,X)$, one has $\mathfrak{L}_{K}(\gamma)=L_{d_K}(\gamma)=A_K(\gamma)$, where $A_K$ is defined by
\[
A_K(\gamma):=\sup \;\sum_{i=0}^{N-1}\; \left(\inf\limits_{t_i\leq t\leq t_{i+1}} K(\gamma(t))\right)\ d(\gamma(t_i),\gamma(t_{i+1}))\in [0,+\infty],
\]
where the supremum is taken over all $N\geq 1$ and all sequences $t_0\leq \dots\leq t_N$ in $I$. Moreover, $A_K$ is invariant under reparametrization (surjective and non decreasing maps), and lower semicontinuous on $\mathcal{C}(I,X)$ endowed with the topology of uniform convergence on compact subsets of $I$.
\end{prop}
\begin{remark}
$\mathfrak{L}_{K}$ is defined on $AC_{ploc}(I,(X,d))$, while $L_{d_K}$ and $A_K$ make sense for any map valued in $X$. Proposition 2 states that all these quantities coincide on the set $AC_{ploc}(I,(X,d))$. An important observation in order to prove that $A_K\leq \mathfrak{L}_{K}$  is the fact that adding a point in the subdivision increases the quantity of which $A_K(\gamma)$ is the supremum. This could have failed if we had defined $A_K$ using $K(\gamma(t))$ instead of $\inf\limits_{t_i\leq t\leq t_{i+1}} K(\gamma(t))$.\end{remark}
\begin{remark}
Proposition \ref{metric} implies in particular that any curve $\gamma\in AC_{ploc}(I,X)$ such that $\mathfrak{L}_{K}(\gamma)<+\infty$ is a continuous function from $I$ onto $(X,d_K)$ which is not obvious since the metric $d$ needs not be stronger than the metric $d_K$.\end{remark}
In order to prove Proposition \ref{metric}, we will need the following elementary estimate:
\begin{lemma}\label{est_dK}
For all $x,y\in X$, one has
\(
K_{d(x,y)}(x)\, d(x,y)\leq d_K(x,y),
\)
where we have set for every $r\geq 0$ and $x\in X$,
\[
K_r(x):=\inf\{K(y)\;:\; d(x,y)\leq r\}.
\]
\end{lemma}
\begin{proof}
Set $r:=d(x,y)$. Since any piecewise locally absolutely continuous curve $\gamma:x\mapsto y$ has to get out of the open ball $B:=B_d(x,r)$, it is clear that its restriction to the part where it stays inside such a ball has at least $d$-length equal to $r$, which gives
\[
\mathfrak{L}_{K}(\gamma)=\int_I K(\gamma(t))\, |\dot{\gamma}|(t)\d t \geq r \inf_B K=rK_r(x).
\]
Taking the infimum over the set of curves $\gamma\in AC_{ploc}$ joining $x$ and $y$ yields the claim. 
\end{proof}

\begin{proof}[Proof of Proposition \ref{metric}] It is quite straightforward that $L_{d_K}$ and $A_K$ are invariant by (surjective and non decreasing) reparametrization and that $\mathfrak{L}_{K}$ is invariant by Lipshitz reparametrizations. Moreover, $\mathfrak{L}_{K}$, $L_{d_K}$ and $A_K$ have the common property that their value on an interval is the supremum of the same value restricted to compact subintervals. Thus one can assume that $I$ is compact and, by affine reparametrization, that $I=[0,1]$ (except if $I$ is reduced to a single point in which case the proposition is trivial). We divide the rest of the proof into five steps.

\medskip
\textsc{$\bullet$ Step 1: $L_{d_K}(\gamma)\leq \mathfrak{L}_{K}(\gamma)$.} This follows from the definitions of $L_{d_K}$ and $d_K$, and from the additivity property of $\mathfrak{L}_{K}$.

\medskip
\textsc{$\bullet$ Step 2: $A_K(\gamma)\leq \mathfrak{L}_{K}(\gamma)$.}
Indeed, by additivity of $\mathfrak{L}_{K}$, the claimed inequality follows from the elementary one,
\[
\left(\inf\limits_{a\leq t\leq b} K(\gamma(t))\right)\ d(\gamma(a),\gamma(b))\leq \mathfrak{L}_{K}(\gamma),
\]
for every curve which is absolutely continuous on the interval $[a,b]$.

\medskip
\textsc{$\bullet$ Step 3: $\mathfrak{L}_{K}(\gamma)\leq A_K(\gamma)$.}
 In order to estimate $A_K(\gamma)$ from below, we use the following subdivision in the interval $[0,1]$: given $n\geq 2$ and $\delta\in [0,\frac 1n]$, define $t^\delta_i=\delta+\frac in$ for $0\leq i\leq n-1$. By definition of $A_K$, one has
\[
A_K(\gamma)\geq \sum_{i=0}^{n-2}\left(\inf\limits_{s\in [t^\delta_i,t^\delta_{i+1}]}K(\gamma(s))\right)\ d\left(\gamma(t^\delta_i),\gamma(t^\delta_{i+1})\right).
\]
Taking the average over $\delta\in [0,\frac 1n]$ yields
\[
A_K(\gamma)\geq\int_0^{\frac 1n} \sum_{i=0}^{n-2}\left(\inf\limits_{s\in [t^\delta_i,t^\delta_{i+1}]}K(\gamma(s))\right)\ \frac{d\left(\gamma(t^\delta_i),\gamma(t^\delta_{i+1})\right)}{\frac 1n}\d \delta.
\]
Since $[0,1-\frac 1n)$ is the disjoint union of the sets $\{t_i^\delta \;:\; 0\leq\delta< 1/n\}$ with $0\le i\le n-2$, we have proved
\[
A_K(\gamma)\geq \int_0^{1-\frac 1n}\left(\inf\limits_{s\in [t,t+\frac 1n]}K(\gamma(s))\right)\ \frac{d\left(\gamma(t),\gamma(t+\frac 1n)\right)}{\frac 1n}\d t.
\]
Now fix $t_0<1$. Since $1-1/n>t_0$ for large $n$, thanks to the lower semicontinuity of $K\circ\gamma$ and Fatou's Lemma, this yields in the $\liminf$ as $n\to \infty$
\[
A_K(\gamma)\geq \int_0^{t_0} K(\gamma(t)) |\dot{\gamma}|(t)\d t
\]
and then it it enough to use the arbitrariness of $t_0$.
\medskip

\textsc{$\bullet$ Step 4: $\mathfrak{L}_{K}(\gamma)\leq L_{d_K}(\gamma)$.} We use the same subdivision $(t_i^\delta)$ of Step 3. We deduce, by Lemma \ref{est_dK} and by definition of $L_{d_K}$, that
\[
L_{d_K}(\gamma)\geq \sum_{i=0}^{n-2} K_{\omega(1/n)}(\gamma(t_i^\delta))\, d(\gamma(t^\delta_i),\gamma(t^\delta_{i+1})),
\]
where $\omega$ is the modulus of continuity of $\gamma$, i.e. $\omega(r)=\sup \{d(\gamma(t),\gamma(s))\;:\; |t-s|\leq r\}$. As in Step 3, taking the average over $\delta\in (0,1/n)$ yields
\[
L_{d_K}(\gamma)\geq \int_0^{1-1/n} K_{\omega(1/n)}(\gamma(t))\frac{d\left(\gamma(t),\gamma(t+\frac 1n)\right)}{\frac 1n}\d t,
\]
and the conclusion follows by taking the $\liminf$ as $n\to\infty$. We use $\liminf_{r\to 0} K_r(x)\geq K(x)$ (since $K$ is l.s.c.) and $\lim_{n\to\infty}\omega(1/n)=0$ (since $\gamma$ is uniformly continuous on $[0,1]$).
\medskip

\textsc{$\bullet$ Step 5: $A_K$ is lower semicontinuous on $\mathcal{C}(I,X)$.} Since the supremum of lower semicontinuous functions is lower semicontinuous and since $\gamma\mapsto d(\gamma(a),\gamma(b))$ is continuous, it is enough to prove that $\gamma\mapsto \inf_{[a,b]} K\circ\gamma$ is lower semicontinuous on the set of curves $\gamma :[a,b]\to X$. Let $(\gamma_n)_n$ be a sequence of curves uniformly converging to a curve $\gamma$ and let $t_n\in [a,b]$ be a point where the infimum of $K\circ\gamma_n$ is achieved. Up to extraction, one can assume that $\inf (K\circ\gamma_n)$ converges to $\liminf_{n\to\infty}(\inf K\circ\gamma_n)$ and that $t_n$ converges to some $t\in [a,b]$ as $n\to\infty$. Then one has $\inf_{[a,b]} K\circ\gamma\leq K(\gamma(t))\leq \liminf_n K(\gamma_n(t_n))$ since $K$ is l.s.c and $\gamma_n(t_n)\to \gamma(t)$.
\end{proof}

We are now ready to prove Theorem \ref{exist_Kgeod}.

\begin{proof}[Proof of Theorem \ref{exist_Kgeod}]
We shall apply Ascoli's Theorem for maps from $[0,1]$ to $(X,d)$. We first explain how to get compactness and then, how to get equicontinuity of a minimizing sequence.

Let $(\gamma_n)_{n\geq 1}\subset AC_{ploc}([0,1],X)$ be a minimizing sequence for the geodesic problem from $x^+$ to $x^-$, i.e. $\gamma_n:x^+\mapsto x^-$ and $\mathfrak{L}_{K}(\gamma_n)\to d_K(x^+,x^-)$ as $n\to\infty$. By Assumption  {\bf (H3)}, we may assume $\gamma_n(t)\in F$ for every $n $ and every $t$. Let $L>0$ be an upper bound for $\mathfrak{L}_{K}(\gamma_n)$, i.e. $\mathfrak{L}_{K}(\gamma_n)\leq L$ for all $n$. Since $Y:=\cup_n \rm{Im}(\gamma_n)$ is included in a set $B\cap F$, where $B$ is a ball for $d_K$, it is $d$-precompact by Proposition \ref{Klength}. Let $\bar Y$ be a $d$-compact set containing $Y$.

In order to get $d$-equicontinuity, one has to reparametrize the curves $\gamma_n$. It may not be possible to parametrize $\gamma_n$ by $L_d$-arc length since $L_d(\gamma_n)$ could be infinite. We rather parametrize by $\mathfrak{L}_{K\wedge 1}$-arc length, where $K\wedge 1$ is the infimum between $K$ and the function identically equal to $1$ (we use $\wedge$ for the minimum operator, and $\vee$ for the maximum). Note that $K\wedge 1$ is still a lower semicontinuous and non negative function vanishing on a finite set so that $d_{K\wedge 1}$ still defines a metric on $X$. Moreover, the curves $\gamma_n$ are continuous from $[0,1]$ to $(X,d_{K\wedge 1})$ since $d_{K\wedge 1}$ is weaker than $d$. Actually, we can prove more: $d_{K\wedge 1}$ and $d$ are topologically equivalent on $\bar Y$. Indeed the identity map, defined from $(\bar Y,d)$ to $(\bar Y,d_{K\wedge 1})$ is a bijective and continuous map defined on a compact set. It is thus a homeomorphism and the two metrics satisfy 
\begin{equation}\label{ddKw1}
d_{K\wedge 1}\leq d\leq \omega(d_{K\wedge 1})
\end{equation}
 for a suitable modulus of continuity $\omega$ (i.e. $\omega:\R_+\to\R_+$ with $\lim_{s\to 0}\omega(s)=0$).

Since $\mathfrak{L}_{K\wedge 1}(\gamma_n)\leq \mathfrak{L}_{K}(\gamma_n)\leq L$, one can reparametrize the curves $\gamma_n$ by constant speed for the distance $d_{K\wedge 1}$. We will call them $\gamma_n:[0,1]\to X$ again. These curves are $L$-Lipschitz w.r.t. $d_{K\wedge 1}$. In particular, the sequence $(\gamma_n)_n$ is equicontinuous w.r.t. the metric $d_{K\wedge 1}$ and for the metric $d$ as well since $d_{K\wedge 1}$ and $d$ are equivalent in the sense of \eqref{ddKw1}. Note that these curves are equicontinuous for the distance $d$, but not equi-Lipschitz, as the modulus of continuity $\omega$ appears.

By Ascoli's Theorem, one can extract a subsequence of $(\gamma_n)_n$ uniformly converging (for the distance $d$) to some continuous map $\gamma:[0,1]\to X$. In particular, $\gamma:[0,1]\to (X,d_{K\wedge 1})$ is $L$-Lipschitz as a pointwise limit of $L$-Lipschitz maps. We have to prove that $\gamma\in\Lip_{ploc}(I,X)$. Note that if there exists $c\in (0,1)$ with $K\geq c$ around a point $x$, then $d\geq d_{K\wedge 1}\geq cd$ on a neighborhood $V$ of $x$ (see Lemma \ref{est_dK}), i.e. $d$ and $d_{K\wedge 1}$ are two equivalent metrics on $V$. As $K\wedge 1$ is l.s.c., it is bounded from below by a positive constant on every compact interval where $K>0$, and it is enough to prove that the set $\mathrm{Im}(\gamma)\cap\Sigma$ is finite. We shall prove that $\gamma$ meets each point of $\Sigma$ at most one time. To this aim, let us consider the curve $\overline{\gamma}$ obtained by inductively withdrawing all loops around a point in the set $\Sigma$: if $\gamma$ meets $x_0\in\Sigma$ at a first time $t_1$ and at a last time $t_2\neq t_1$, remove $\gamma_{(t_1,t_2)}$ and rescale (using an affine change of variables) for it to be still defined on $[0,1]$. Repeating this operation inductively for each point in the set $\mathrm{Im}(\gamma)\cap\Sigma$ (which is finite), we obtain a curve $\overline{\gamma}$ which is $d_{K\wedge 1}$-Lipschitz and which meets $\Sigma$ a finite number of times. By the argument above, this implies that $\overline{\gamma}\in \Lip_{ploc}([0,1],X)$. Moreover, Proposition \ref{metric} implies that
\[
\mathfrak{L}_{K}(\overline{\gamma})=A_K(\overline{\gamma})\leq A_K(\gamma)\leq\liminf\limits_{n\to\infty} A_K(\gamma_n)=\liminf\limits_{n\to\infty} \mathfrak{L}_{K}(\gamma_n).
\]
As $(\gamma_n)_n$ is a minimizing sequence, all these inequalities are in fact equalities and $\overline{\gamma}$ minimizes the $K$-length between $x$ and $y$. Note that the saturation of the first inequality above also implies that $\overline{\gamma}=\gamma$.
\end{proof}

\section{Existence of heteroclinic connections\label{sectionHeteroclinic}}
Our aim is to investigate the existence of a global minimizer (called heteroclinic connection in the sequel) of the energy
\[
\mathfrak{E}_{W}(\gamma)=\int_\R \bigg(\frac{1}{2}|\dot{\gamma}|^2(t)+W(\gamma(t))\bigg)\d t ,
\]
among piecewise locally absolutely continuous curves $\gamma: x^-\mapsto x^+$ valued in a metric space $(X,d)$. Here $W: X\to\R^+$ is a lower semicontinuous function, called \emph{potential} in all the sequel, and $x^\pm\in X$ are two wells, i.e. $W(x^\pm)=0$. We recall that $W(x^\pm)=0$ is a necessary condition for the energy of $\gamma$ to be finite. The main result of this section is the following theorem:
\begin{theorem}\label{heteroclinic}
Let $(X,d)$ be a metric space, $W:X\to\R^+$ be a lower semicontinuous function and $x^-,x^+$ be two points in $\Sigma:=\{W=0\}\subset X$ such that:
\begin{description}
\item[(H)]
$(X,d,K)$ satisfies hypotheses ${\bf H1-3}$ of the previous section, where $K:=\sqrt{2W}$;
\item[(STI)]
the metric $d_K$ satisfies the following strict triangular inequality on $\Sigma:=\{W=0\}$:  for all $x\in \Sigma\setminus\{x^-,x^+\}$, we have $d_K(x^-,x^+)<d_K(x^-,x)+d_K(x,x^+)$.
\end{description}
Then, if $d_K(x^-,x^+)<\infty$, there exists $\gamma\in \Lip(\R,X)$ such that $\gamma:x^-\mapsto x^+$ and
\[
\mathfrak{E}_{W}(\gamma)=\inf\{\mathfrak{E}_{W}(\sigma)\;:\; \sigma\in AC_{ploc}(\R,X),\, \sigma:x^-\mapsto x^+\}=d_K(x^-,x^+).
\]
\end{theorem}
\begin{proof}
This theorem will follow from Theorem \ref{exist_Kgeod} and the following consequence of Young's inequality (keep in mind that $K=\sqrt{2W}$):
\begin{equation}\label{young}
\forall\gamma\in AC_{ploc}(\R,X),\quad \frac{1}{2}|\dot{\gamma}|^2(t)+W(\gamma(t))\geq K(\gamma(t))|\dot{\gamma}|\,\Rightarrow\,\ \mathfrak{E}_{W}(\gamma)\geq \mathfrak{L}_{K}(\gamma).
\end{equation}
The idea is to build the curve $\gamma$ by reparametrization of a $\mathfrak{L}_{K}$-minimizing curve in such a way that the preceding inequality is an equality. Thanks to the set of assumptions {\bf (H)}, Theorem \ref{exist_Kgeod} provides a $\mathfrak{L}_{K}$-minimizing curve $\gamma_0 : I=(t^-,t^+)\to X$, that one can assume to be injective and parametrized by $\mathfrak{L}_{K}$-arc length, and such that $\gamma_0(t^\pm )=x^\pm$ with $-\infty<t^-\leq t^+<+\infty$. Thanks to assumption ${\bf (STI)}$, it is clear that the curve $\gamma_0$ cannot meet the set $\{W=0\}$ at a third point $x\neq x^\pm$: in other words $K(\gamma_0(t))>0$ on the interior of $I$. Thus $\gamma_0$ is also $d$-locally Lipschitz on $I$ (and not only piecewise locally Lipschitz). In particular, one can reparametrize the curve $\gamma_0$ by $L_d$-arc length, so that $|\dot{\gamma_0}|=1$ a.e. 

It is now enough to prove that $\gamma_0$ can be reparametrized into a curve $\gamma$ satisfying $|\dot{\gamma}|=K\circ\gamma$ a.e., so that the two equalities in \eqref{young} become equalities. Namely, we look for an admissible curve $\gamma:\R\to X$ of the form $\gamma(t)=\gamma_0(\varphi(t))$, where $\varphi:\R\to I$ is absolutely continuous, increasing and surjective. For $\gamma$ to satisfy the equipartition condition, i.e. $|\dot{\gamma}|(t)=K(\gamma(t))$ a.e., we need $\varphi$ to solve the ODE
\begin{equation}\label{ode_geod}
\varphi'(t)=F(\varphi(t)),
\end{equation}
where $F:I\to\R$ is defined by $F=K\circ\gamma_0$. Here, one has to be careful since $F$ is not continuous and the existence of such a curve does not follow from Peano-Arzela's theorem. Actually, the best regularity on $\varphi$ that can be expected is absolute continuity. However, the situation is very simple here and one can explicitly solve this scalar ODE, at least formally. Indeed, \ref{ode_geod} is equivalent to $[G(\varphi(t))]'=1$ where $G$ is an antiderivative of $1/F$. Thus a solution is given by $\varphi=G^{-1}$. A rigorous statement about the existence of a solution is given in Lemma \ref{ode}, which is presented a the end of this proof.
It is easy to check that $F=K\circ\gamma_0$ satisfies all the assumptions of Lemma \ref{ode}. The condition $F<\infty$ a.e. is a consequence of $\mathfrak{L}_{K}(\gamma_0)=\int K(\gamma_0)|\dot{\gamma_0}|=\int F<\infty$. 

Now, let us define $\gamma:=\gamma_0\circ\varphi$, where $\varphi$ is given by Lemma \ref{ode}. Since $\varphi$ is increasing and surjective, the curve $\gamma$ satisfies $\gamma(\pm\infty)=x^\pm$. Moreover, $\gamma$ is absolutely continuous, by composition of a Lipschitz function with an absolutely continuous function, and its pointwise derivative is given by $|\dot{\gamma}|=|\dot{\gamma_0}|\, \varphi'=F(\varphi)=K\circ\gamma$ a.e. As explained before, this implies that
\[
\mathfrak{E}_{W}(\gamma)=\mathfrak{L}_{K}(\gamma)=\mathfrak{L}_{K}(\gamma_0)=d_K(x^-,x^+)\leq\inf\{\mathfrak{E}_{W}(\sigma)\;:\; \sigma\in AC_{ploc}(\R,X),\,\gamma:x^-\mapsto x^+\}.
\]
In other words, $\gamma$ minimizes $\mathfrak{E}_{W}$ over all admissible connections between $x^-$ and $x^+$. 
\end{proof}

In the proof of Theorem \ref{heteroclinic} we needed this very technical lemma.
\begin{lemma}\label{ode}
Let $F:I\to (0,+\infty]$ be a l.s.c. function defined on a nonempty open interval $I\subset\R$, with $F<+\infty$ a.e. Then there exists an interval $J\subset\R$ and a surjective function $\varphi\in AC(J,I)$ solving the system
\[
\varphi'(t)=F(\varphi(t))\quad\text{a.e. on }J.
\]
Moreover, $1/F$ is locally bounded and a solution of the above equation is given by $\varphi=G^{-1}$, where $G$ is an antiderivative of $1/F$.
\end{lemma}

\begin{proof}
As $F$ is l.s.c. and $F>0$ on $I$, it is bounded from below by a positive constant on every compact subset of $I$. Thus $1/F$ is locally bounded and positive a.e. (as $F<\infty$ a.e.). In particular, $G:I\to J:=G(I)$ is locally Lipschitz and strictly increasing. Thus $\varphi:=G^{-1}:J\to I$ is well defined, strictly increasing, surjective, and thus continuous. One has to prove that $\varphi$ is absolutely continuous and that $\varphi'(t)=F(\varphi(t))$ a.e. To this aim, we first approximate $F$ as an increasing limit of step functions $(F_n)_n$. This is possible because $F$ is l.s.c., taking for instance a dyadic subdivision on intervals of the form $(k2^{-n},(k+1)2^{-n}]$ and defining $F_n$ as the infimum of $F$ on each of these intervals. Let us call $t^k_n$ the endpoints of these intervals and $\lambda^k_n$ the value of $F_n$ on $(t^k_n,t^{k+1}_n]$. Let us take $a\leq b$ in $J$, and compute
\[
\int_a^b F_n(\varphi(t))\d t=\sum_k \lambda^k_n \mathcal{L}^1\big(\{t\,:\, a\leq t\leq b, \, t^k_n\leq \varphi(t)\leq t^{k+1}_n\}\big)=\sum_k \lambda^k_n\big|(G(t^{k+1}_n)\wedge b)-(G(t^k_n)\vee a)\big|.
\]
Using $b=G(\varphi(b))$, $a=G(\varphi(a))$ and $G'=1/F$, we can go on and obtain
\[\int_a^b F_n(\varphi(t))\d t=\sum_k \lambda^k_n\int_{t^k_n\vee \varphi(a)}^{t^{k+1}_n\wedge \varphi(b) }\frac{1}{F(s)}\d s=\int_{\varphi(a)}^{\varphi(b) }\frac{F_n(s)}{F(s)}\d s.
\]
By the monotone convergence Theorem, passing to the limit $n\to\infty$, we get
\[\int_a^b F(\varphi(t))\d t=\varphi(b)- \varphi(a).\]
This tells us that $F\circ\varphi$ is integrable, $\varphi$ is absolutely continuous, and that its derivative is given by $\varphi'=F\circ\varphi$ a.e., which is the claim.
\end{proof}

\begin{remark}
It is classical and easy to see that the equirepartition of the energy, i.e. the identity $|\dot{\gamma}|^2(t)=2W(\gamma(t))$, is a necessary condition for critical points of $\mathfrak{E}_{W}$.
\end{remark}
\begin{remark}
The assumption {\bf (STI)} is not optimal but cannot be removed, and is quite standard in the literature. Without this assumption, it could happen that a geodesic $\gamma_0$ would meet the set $\{W=0\}$ at a third point $x\neq x^\pm$. In this case, it is not always possible to reparametrize $\gamma_0$ in a new curve $\gamma$ such that $|\dot{\gamma}|(t)=K(\gamma(t))$. 

 
However, there is an easy generalization of Lemma \ref{ode} when $F$ is not positive everywhere but only almost everywhere and when $1/F$ is locally integrable on $I$. Thus, if we assume that $1/K(\gamma_0(\cdot))$ is locally integrable (which is an assumption on the way $K$ vanishes around its wells), we still have existence of an $\mathfrak{E}_W$-minimizing curve between $x^-$ and $x^+$ obtained by reparametrization of $\gamma_0$.
\end{remark}
We now give a first example of application of Theorem \ref{heteroclinic} in infinite dimension:
\begin{corollary}\label{pde}
Let $m,n\ge 1$ be integers, $\Omega\subset\R^m$ be a bounded open set, $F:\Omega\times\R^n\to\R$ be a lower semicontinuous function and $g$ be a function in $H^1(\Omega;\R^n)$ such that the problem
$$\min\left\{\W(u):=\int_\Omega \left(\frac 12|\nabla u|^2+F(x,u(x))\right)dx \;:\;u-g\in H^1_0(\Omega,\R^n)\right\}$$
admits exactly two distinct solutions $u^\pm\in H^1(\Omega;\R^n)$ and the minimal value is $0$. Also suppose that there exist three constants $c_0,c_1,c_2,c_3>0$ such that
\begin{itemize}
\item for every $x\in\Omega$ and $u\in\R^n$, $F(x,u)\geq -c_0-c_1|u|^2$,
\item for every $u\in H^1(\Omega;\R^n)$ with $\|u\|_{L^2(\Omega)}\geq c_2$, we have $\W(u)\geq c_3$.
\end{itemize}
Then the following minimization problem has a solution:
$$
\min \left\{\int_{\R\times\Omega}\left( \frac{1}{2}|\nabla u|^2(x)+F(x',u(x))\right)\d x\;:\; u\in H^1_{loc}(\R\times\Omega,\R^n)\text{ with }u(\pm\infty,\cdot)=u^\pm,\, u=g\;\mbox{ on }\partial\Omega\right\},
$$
where we write $x=(x_1,x')\in\Omega$ and the boundary condition, $u(\pm\infty,\cdot)=u^\pm$, means that $u(x_1,\cdot)\to u^\pm$ in $L^2(\Omega)$ as $x_1\to\pm\infty$ and the lateral boundary condition $u=g$ means that $u(x_1,\cdot)-g(\cdot)\in H^1_0(\Omega)$ for a.e. $x_1\in \R$.\end{corollary}
\begin{remark} The lower bound $\W(u)\geq c_3$ for $\|u\|_{L^2(\Omega)}\geq c_2$ is a strong assumption, corresponding to taking a constant function $k$ in Assumption \textbf{(H3a)}. It could be weakened, but it is enough for many applications. It is satisfied whenever $F$ is bounded from below, or if $c_0,\, c_1$ are small enough, for instance.\end{remark}
\begin{remark} The assumption $\min \W=0$ can always be enforced by subtracting a suitable constant to the function $F$.\end{remark}

\begin{remark}
		 Our boundary constraint is not equivalent to the pointwise boundary condition $u(x_1,x_2)\to u_\pm$ as $x_1\to\pm\infty$, for a.e. $x_2$. Indeed, $L^2$-convergence only implies convergence a.e. of a subsequence, and pointwise convergence needs not imply $L^2$-convergence without additional assumption. However, if $u\in H^1_{loc}$ has finite energy, then these two notions can be proven to be equivalent.
\end{remark}
\begin{proof}
	The main idea is to rewrite the total energy $E(u):=\int_{\R\times\Omega}\left( \frac{1}{2}|\nabla u|^2(x)+F(x',u(x))\right)\d x$ by separating the derivative in the first variable $x_1$ from the derivatives in $x'\in\Omega$:
	\[
	E(u)=\int_\R \left(\frac 12\|\partial_{x_1} u\|^2_{L^2(\Omega)}+\int_{\Omega} \left(\frac 12 |\nabla_{x'}u|^2+F(x',u)\right)\d x'\right)\d x_1.
	\]
	Thus this problem fits into our framework using $x_1$ as time (or parameter) variable $t$, $X=L^2(\Omega,\R^n)$ for the metric space (endowed with the $L^2$-distance), with the effective potential $\W:L^2(\Omega,\R^n)\to \R$
	 defined by
	 \[
	 \W(v):=
	 \begin{cases}
	 \int_\Omega \left(\frac 12 |\nabla v|^2 +F(x',v)\right)\d x'&\text{if }v-g\in H^1_0(\Omega,\R^n),\\
	 +\infty&\text{otherwise}.
	 \end{cases}
	 \]
	 Indeed, there is a canonical one-to-one correspondence between $L^2_{loc}(\R\times\Omega,\R^n)$ and $L^2_{loc}(\R,L^2(\Omega,\R^n))$: any function $(x_1,x')\mapsto u(x_1,x')$ is associated with the curve $x_1\mapsto u(x_1,\cdot)$ and vice-versa. In this correspondence (we will not distinguish these two objects in the sequel), $H^1_{loc}(\R\times\Omega,\R^n)$ is contained in $AC_{loc}(\R,L^2(\Omega,\R^n))$. Moreover, for every $u\in H^1_{loc}(\R\times\Omega,\R^n)$, the metric derivative of $t\mapsto u(t):=u(t,\cdot)\in L^2(\Omega,\R^n)$ is given by
	 \[
	 |\dot{u}|_{L^2(\Omega,\R^n)}(t)=\|\partial_{x_1}u(t,\cdot)\|_{L^2(\Omega,\R^n)}\text{ a.e.}
	 \]	 
	 We have to check that $(L^2(\Omega,\R^n),\W)$ satisfies all the assumptions of Theorem \ref{heteroclinic}. It is easy to see that $\W$ is lower semicontinuous for the strong convergence of $L^2(\Omega,\R^n)$: the first term, $\|\nabla v\|^2_{L^2}/2$,  is a standard convex functional of calculus of variations, and the second term, $\int F(x,u)$, can be dealt with Fatou's Lemma, after adding a term of the form $c_0+c_1|u|^2$ so as to make it positive (and this quadratic term is continuous for the strong $L^2$ convergence). Moreover, we exactly assumed that $\W$ has only two wells corresponding to the optimal functions $u^\pm$. In particular, Assumption \textbf{(STI)} is empty.
	 
	 Assumption \textbf{(H3b)} is clearly satisfied since sublevel sets of $\W$ intersected with a ball in $L^2(\Omega,\R^n)$ are bounded in $H^1(\Omega,\R^n)$ (using the quadratic lower bound on $F$) and hence compact in $L^2(\Omega,\R^n)$. Assumption \textbf{(H3a)} is a consequence of the lower bound $\W(u)\geq c_3$ for $\|u\|_{L^2}\geq c_2$. Corollary \ref{pde} is now a consequence of Theorem \ref{heteroclinic}.
\end{proof}

We can consider for instance the following, non-trivial example. We look for a solution of 
$$
\min \left\{\int_{\R\times [0,\pi]}\left( \frac{1}{2}|\nabla u|^2(x,y)-\frac{1}{2}|u|^2(x,y)+(u^2(x,y)-\sin^2(y))^2\right)\d x\d y\;:\; \begin{array}{l} u\in H^1_{loc}(\R\times[0,\pi];\R),\\u(\pm\infty,y)=\pm \sin(y),\\ u(x,0)=u(x,\pi)=0.\end{array}\right\}.
$$
The problem is non-trivial and provides a solution to the PDE
$$\Delta u(x,y)=-u(x,y)+4u(x,y)(u^2(x,y)-\sin^2(y)),$$
which does not seem easy to solve with the prescribed boundary condition (in particular, the solution is not of the form $u(x,y)=a(x)\sin (y)$). This example can be dealt with the above formalism, noting that $u^\pm(y)=\pm\sin(y)$ are the only two functions for which 
$$\W(v):=\int_0^\pi \left(\frac 12 (|v'(y)|^2-v^2(y))+(v^2(y)-\sin^2(y))^2\right)\d y=0$$
(indeed, the first term is non-negative by using the optimal constant in the Poincar\'e inequality in $H^1_0([0,\pi])$, but vanishes for functions of the form $v(y)=a\sin(y)$, and the second only vanishes if $a=\pm 1$). Moreover, $F(y,u):=-\frac 12 u^2+(u^2-\sin^2(y))^2\geq u^4-C(1+u^2)$ satisfies the required lower bounds so as to apply Corollary \ref{pde}.

\section{Stationary layered solutions for the Allen-Cahn system in two dimensions \label{sectionDouble}}

\subsection{Introduction}

Let $n\ge 1$, $W:\R^n\to\R^+$ be a potential and $a^\pm\in\R^n$ such that $W(a^\pm)=0$. Assume that
\begin{description}
	\item[(A1)]
	$W\in\mathcal{C}^2(\R^n,\R^+)$ and the Hessian $\nabla^2W$ is bounded from below (i.e. $W$ is semi-convex) on $\R^n$;
	\item[(A2)]
	 $\Sigma:=\{W=0\}$ is finite;
	\item[(A3)]
	for all $x\in\R^n$, $W(x)\geq k^2(d(x,\Sigma))$ for some function $k\in C^0(\R^+,\R^+)$ with $\int_0^\infty k(t)\d t=+\infty$;
	\item[(A4)]
	there exist $r_0>0$, $c_0>0$ and $p_0\in [2,6)$ such that for all $x\in B(a^\pm,r_0)$, $$\nabla W(x)\cdot (x-a^\pm)\geq c_0|x-a^\pm|^{p_0}.$$
\end{description}
As before, we also need the following assumption that avoids heteroclinic connections to meet a third point of the potential $W$:
\begin{description}
\item[(STI)]
for all $a\in \Sigma\setminus\{a^-,a^+\}$,\, $d_K(a^-,a^+)<d_K(a^-,a)+d_K(a,a^+)$,
\end{description}
where $K:=\sqrt{2W}$ and $d_K$ was defined in \eqref{Kdistance}. We investigate the following problem: assuming that there exist exactly two (up to translation) heteroclinic connections $z^-$ and $z^+$  between $a^-$ and $a^+$, does there exist a solution $u\in C^{2}(\R^2,\R^n)$ of the following system
\begin{equation}\label{Double_connections}
\begin{cases}
-\Delta u+\nabla W(u)=0&\text{over }\R^2;\\
u(x_1,x_2)\to a^-&\text{when }x_1\to -\infty,\text{ uniformly w.r.t. }x_2;\\
u(x_1,x_2)\to a^+&\text{when }x_1\to \infty,\text{ uniformly w.r.t. }x_2;\\
u(x_1,x_2)\to z^-(x_1-c^-)&\text{when }x_2\to -\infty,\text{ uniformly w.r.t. }x_1;\\
u(x_1,x_2)\to z^+(x_1-c^+)&\text{when }x_2\to \infty,\text{ uniformly w.r.t. }x_1;
\end{cases}
\end{equation}
where $c^-$ and $c^+$ are part of the unknown. We will see that our method allows to treat quite easily the problem of the existence of a solution in two known situations: the symmetric case (we add a symmetry condition on $z^-$, $z^+$ and $u$ which imposes, in particular, $c^-=c^+=0$), studied by Stanley Alama, Lia Bronsard and Changfeng Gui \cite{Alama:1997} (we will deal with this problem in Section \ref{sectionAlama} of the present paper), and the asymmetric case due to Michelle Schatzman \cite{Schatzman:2002} (in Section \ref{sectionSchatzman}). We start with some preliminary results valid in a general context (without the symmetry condition of \cite{Alama:1997} or the spectral condition of \cite{Schatzman:2002}).

\subsection{Action functional and heteroclinic connections}

As before, define the energy functional $\mathfrak{E}_{W}$ for all interval $I\subset \R$ and for all ${v}\in H^1_{loc}(I,\R^n)$ by
\[
\mathfrak{E}_{W}({v},I)=\int_I \left(\frac 12 |\dot{{v}}(t)|^2+W({v}(t))\right)\d t.
\]
If $I=\R$, we just write $\mathfrak{E}_{W}({v},\R)=:\mathfrak{E}_{W}({v})$. The set of minimizing heteroclinic connections between $a^-$ and $a^+$ is given by
\[
\mathcal{Z}:=\{{z}\in\mathcal{S}(a^-,a^+)\;:\; \forall {v}\in\mathcal{S}(a^-,a^+),\, \mathfrak{E}_{W}({z})\leq \mathfrak{E}_{W}({v})\},
\]
where $\mathcal{S}(a^-,a^+)$ stands for the set of all connections between $a^-$ and $a^+$: in general, $\mathcal{S}(b^-,b^+)$ is defined for two given points $b^\pm\in\R^n$ by
\[
\mathcal{S}(b^-,b^+)=\left\{{v}\in H^1_{loc}(\R,\R^n)\;:\; \lim\limits_{t\to\pm\infty}{v}(t)=b^\pm\right\}.
\]
As observed in Theorem \ref{heteroclinic}, one can identify the minimal value of $\mathfrak{E}_{W}$ on $\mathcal{S}(a^-,a^+)$:
\[
\inf\{\mathfrak{E}_{W}({v})\;:\;{v}\in\mathcal{S}(a^-,a^+)\}=d_K(a^-,a^+),
\]

\subsection{Variational formulation of \eqref{Double_connections}} In \cite{Alama:1997,Schatzman:2002}, it is shown that solutions of \eqref{Double_connections} can be found by minimizing the renormalized two-dimensional energy
\[
\mathcal{E}({u})=
\begin{cases}
\int_\R\left[\int_\R\left(\frac 12|\nabla {u}(x_1,x_2)|^2+W({u}(x_1,x_2))\right)\d x_1-d_K(a^-,a^+)\right]\d x_2&\text{if }{u}\in H^1_{loc}(\R^2,\R^n),\\
+\infty &\text{otherwise.}
\end{cases}
\]
Since $d_K(a^-,a^+)$ is the minimal value of $\mathfrak{E}_{W}$ on $\mathcal{S}(a^-,a^+)$, it is clear that $\mathcal{E}$ is nonnegative. We recall the following well-known fact:
\begin{prop}
Assume that ${u}\in H^1_{loc}\cap L^\infty_{loc}(\R^2,\R^n)$ has finite energy and locally minimizes $\mathcal{E}$ in the following sense
\[
\mathcal{E}({u})\leq \mathcal{E}({u}+{w})\quad\text{for every smooth compactly supported function }{w}:\R^2\to\R^n.
\]
Then ${u}$ solves the Euler-Lagrange equation $-\Delta {u}+\nabla W({u})=0$ (in the weak sense).
\end{prop}
A crucial observation is that we can interpret $\mathcal{E}$ as the action of a curve $x_2\mapsto {u}(\cdot,x_2)$, plotted on a subset of $L^2_{loc}(\R,\R^n)$. Indeed, the energy of ${u}$ rewrites
\begin{equation}
\label{energy_length}
\mathcal{E}({u})=\int_\R \left(\frac 12\|\partial_{x_2}{u}(\cdot,x_2)\|_{L^2(\R)}^2 +\mathcal{K}({u}(\cdot,x_2))^2\right)\d x_2, \quad \text{where }\mathcal{K}({v}):=\sqrt{\mathfrak{E}_{W}({v})-d_K(a^-,a^+)}.
\end{equation}
Moreover, the first term, $\|\partial_{x_2}{u}(\cdot,x_2)\|_{L^2(\R)}$, is nothing but the metric derivative (for the $L^2$-distance) of the curve $x_2\mapsto {u}(\cdot,x_2)$. Indeed, the following fact can be easily proven:
\begin{lemma}\label{l2curve}
Let $X$ be a subspace of $L^2_{loc}(\R,\R^n)$ endowed with the $L^2$-distance: $d_X(v_1,v_1):=\|v_1-v_2\|_{L^2(\R)}$. Let $\gamma:I\to X$ be a curve in $X$: for all $t\in I$, $\gamma(t)$ is a function defined on $\R$ and we use the notation $\gamma(t,s):=\gamma(t)(s)$ for every $t\in I$ and $s\in\R$. Then the function $\gamma(\cdot,\cdot)$ is measurable and belongs to $L^2_{loc}(\R^2,\R^n)$. 

Moreover, if $\gamma\in AC_{ploc}(I,X)$, then $\partial_t\gamma(t,\cdot)\in L^2(\R,\R^n)$ for a.e. $t\in I$ and the metric derivative of $t\mapsto \gamma(t)=\gamma(t,\cdot)$ in $X$ is given by
\[
|\dot{\gamma}|(t)=\|\partial_t \gamma(t,\cdot)\|_{L^2(\R)}\quad a.e.
\]
\end{lemma}

\subsection{Projecting on ``decreasing funnels'' reduces the energy} It is well known that every heteroclinic connection converges exponentially fast to its limits at $\pm\infty$, at least with a non degeneracy assumption on $W$, i.e. $p_0=2$ in $\mathbf{(A4)}$. This can be easily proved by a maximum principle. In order to get enough compactness to prove the existence of double heteroclinic connections, we need a more precise result. Namely, we prove that the energy is reduced when $u$ is projected on a set of functions -- that we call funnel because of its shape -- of the form $\{v:\R\times\R^2\;:\;|v(t)-a^+|\leq E(t)\}$ with $E(t)\to 0$ as $t\to +\infty$. In paritcular, this will be general enough to handle the case of a degeneracy fir $p_0\neq 2$ in {\bf (A4)}.
 
\begin{lemma}\label{projection}
There exists $\varepsilon_0\in (0,1)$ only depending on $W$, $c_0$ and $p_0$ (given in $\mathbf{(A4)}$) such that the following holds true. Let ${v}\in H^1_{loc}(\R,\R^n)$ satisfy $\mathfrak{E}_{W}({v})<\infty$ and ${v}(s_0)\in B(a^+,\varepsilon_0)$ for some $s_0\in \R$. Take $c\in (0,c_0)$. Then one has $\mathfrak{E}_{W}(P[{v}])\leq \mathfrak{E}_{W}({v})$, where
\[
P[{v}](s):=
\begin{dcases}
a^++E(s)\, \frac{{v}(s)-a^+}{|{v}(s)-a^+|}&\text{if }s>s_0\text{ and }|{v}(s)-a^+|>E(s),\\
{v}(s)&\text{otherwise,}
\end{dcases}
\]
and where $E:[s_0,+\infty)\to \R^+$ is solution of the following system:
\[
\begin{cases}
E''=cE^{p_0-1},\\
E(s_0)=\varepsilon_0,\\
\lim\limits_{s\to +\infty}E(s)=0.
\end{cases}
\]
\end{lemma}
\begin{remark}
$P[{v}]$ is a projection of ${v}$ in $L^2_{loc}(\R,\R^n)$ endowed with the $L^2$-distance (which of course is not finite everywhere in $L^2_{loc}$) onto the convex set $\mathcal{C}^+\subset L^2_{loc}(\R,\R^n)$ defined by
\[
\mathcal{C}^+:=\{{v}\in L^2_{loc}(\R,\R^n)\;:\; \text{for a.e. } s\geq s_0,\, |{v}(s)-a^+|\leq E(s)\}.
\]
In other words, one has $P[{v}]\in \mathcal{C}^+$ and $\|{v}-P[{v}]\|_{L^2}\leq \|{v}-\overline{v}\|_{L^2}$ for all $\overline{v}\in \mathcal{C}^+$. By convexity, $P[{v}]$ is the unique projection of ${v}$ in the preceding sense, at least if $\|{v}-P[{v}]\|_{L^2}<\infty$.
\end{remark}
\begin{remark}\label{explicit}
There is an explicit expression for $E$:
\[
E(s)=
\begin{dcases}
\varepsilon_0\exp{\left(-\sqrt{c}\, (s-s_0)\right)}&\text{if }p_0=2,\\
\frac{c(s-s_\ast)^{-\alpha}}{\alpha(\alpha+1)}&\text{if }p_0\in (2,6),
\end{dcases}
\]
where $\alpha>1/2$ is determined by $p_0=2+\frac{2}{\alpha}$, and $s_\ast<s_0$ is determined by $(s_0-s_\ast)^{-\alpha}=c^{-1}\varepsilon_0\alpha(\alpha+1)$.
It will be useful for the next computations to observe that we have $E\in L^2$ for $\alpha>1/2$ (i.e. $p_0<6$, which explains why we limited ourselves to such a case) and
$$|E'(s_0)|=C(c,p_0)\varepsilon_0^{p_0/2},\qquad \int_{s_0}^\infty |E(s)|^2\d s=C(c,p_0)\varepsilon_0^{3-p_0/2}.$$
\end{remark} 
\begin{proof}
One can assume that $a^+=0$. If the conclusion of the lemma holds false, then there exists at least one connected component $I$ in $(s_0,+\infty)\cap\{|{v}|>E\}$ such that $\mathfrak{E}_{W}(P[{v}],I)>\mathfrak{E}_{W}({v},I)$. Let us write $I=(s^-,s^+)$, with $s_0\leq s^-<s^+\leq +\infty$. Since $P[{v}](s_0)={v}(s_0)\in B(a^+,\varepsilon_0)$ and $E(s_0)=\varepsilon_0$, one has actually $s^->s_0$. Moreover, by construction, one has
\[
\begin{cases}
|{v}(s^-)|=E(s^-),&\\
|{v}(s^+)|=E(s^+)&\text{if }s^+<+\infty,\\
|{v}(s)|>E(s)&\text{for all }s\in I.
\end{cases}
\]
Define $\sigma$ by $\sigma(s)=|{v}(s)|^{-1}{v}(s)$ for all $s\in I$, so that $P[{v}](s)=E(s)\sigma(s)$ on $I$. We will have a contradiction if we prove that $f=E$ is solution of the following minimization problem:
\begin{equation}
\label{minf}
\min\left\{
\mathfrak{E}_{W}(f\sigma,I)\;:\; f\in H^1(I),\, f\geq E,\, f(s^-)=E(s^-),\text{ and }f(s^+)=E(s^+)\text{ if }s^+<+\infty\right\}.
\end{equation}
First note that, with our assumptions, the minimum of the energy is finite. Indeed, for all $f\in H^1(I)$, one has
\begin{align*}
\mathfrak{E}_{W}(f\sigma,I)= \int_I \frac 12(f')^2+W(f\sigma)+\frac 12|\sigma'|^2 f^2,
\end{align*}
which is finite in particular when $f=E$. Indeed, we have on the one hand 
\[
\int_I|\sigma'|^2 E^2=\int_I\frac{E^2}{|{v}|^2}\left({v}'-({v}'\cdot \frac{{v}}{|{v}|}) \frac{{v}}{|{v}|}\right)^2\leq 4\int_I|{v}'|^2\leq 8\mathfrak{E}_{W}({v})<+\infty. 
\]
On the other hand, since $E''E\ge 0$ and $W(z)\leq C_0 |z|^2$ whenever $|z|\leq \varepsilon_0\leq 1$ (this comes from the fact that $W$ is $\mathcal{C}^2$ and $W$, $\nabla W$ vanish at $z=a^+=0$), we have
 \begin{align*}
 \int_{s_0}^{\infty} \frac 12(E')^2+W(E\sigma)&\leq -\frac 12E'(s_0)E(s_0)-\frac 12\int_{s_0}^{+\infty}E''E+\int_{s_0}^{+\infty}W(E\sigma)\\
 &\leq -\frac {\varepsilon_0}{2}E'(s_0)+C_0\int_{s_0}^{+\infty}E^2.
\end{align*}
Using the explicit formula for $E$ (see Remark \ref{explicit}), we see that the right hand side is finite. 

Now, by a standard application of the direct method in the Calculus of variations, there exists a minimizer $f\in H^1(I)$ of the problem \eqref{minf}. In particular, $f$ solves the Euler-Lagrange equation
\begin{equation}\label{euler_lagrange}
-f''+f|\sigma'|^2+\nabla W(f\sigma)\cdot\sigma=0\quad \text{on the open set }J:=\{f>E\}.
\end{equation}
Moreover, by minimality, we know that $\mathfrak{E}_{W}(f\sigma,I)\leq \mathfrak{E}_{W}(E\sigma,I)$. Together with the inequality $f\geq E$ and the identity $\sigma'\cdot\sigma=0$, this yields
\begin{align*}
 \int_I \frac 12(f')^2+W(f\sigma)= \mathfrak{E}_{W}(f\sigma,I)-\frac 12\int_I (\sigma'f)^2\leq \mathfrak{E}_{W}(E\sigma,I)-\frac 12\int_I |\sigma'|^2 E^2= \int_{s_0}^{+\infty} \frac 12(E')^2+W(E\sigma).
 \end{align*}
By the previous estimates, the right hand side is controlled by $-\frac {\varepsilon_0}{2}E'(s_0)+C_0\|E\|_{L^2((s_0,+\infty))}^2$ which tends to $0$ with $\varepsilon_0$ (see again the explicit computations of Remark \ref{explicit}). 

Now, we want to combine the Euler-Lagrange equation \eqref{euler_lagrange} and our assumption $\mathbf{(A4)}$. We need to prove that, for $\varepsilon_0$ small enough, one has $f<r_0$, the constant appearing in $\mathbf{(A4)}$. This is a consequence of the estimation of $\int_I \frac 12(f')^2+W(f\sigma)$ and the inequality
\[
\mathfrak{L}_{K}(f):=\int_I \sqrt{2W(f(s)\sigma(s))}|f'(s)|\d s\leq\int_I \frac 12(f'(s))^2+W(f\sigma(s))\d s,
\]
where $\mathfrak{L}_{K}$ is nothing but the $K$-length in $\R$ endowed with the weight function $K(f)=\sqrt{2W(f\sigma)}$. Thus, as in the proof of Proposition \ref{Klength}, one sees that $f(s)$ stays in a ball centered at $f(s^-)=E(s^-)<\varepsilon_0$, and whose radius tends to $0$ with $\mathfrak{L}_{K}(f)$. In particular, there exists a constant $\varepsilon_0^1>0$ depending on $p_0$, $c_0$, $r_0$ and $W$ such that for $\varepsilon_0<\varepsilon_0^1$, one has 
\[
f(s)<r_0\quad\text{for all }s\in I.
\]
Now, for such small values of $\varepsilon_0$, \eqref{euler_lagrange} and $\mathbf{(A4)}$ provide the estimate
\[
\frac{f''}{f}=|\sigma'|^2+\frac{\nabla W(f\sigma)\cdot f\sigma}{f^2}\geq c_0 f^{p_0-2}\quad \text{on }\{f>E\}.
\]
If $f=E$, there is nothing to prove. Otherwise, $\{f>E\}$ contains a non empty connected component $J\subset I$.

First assume that $J$ is bounded: then $f>E$ on $J$ and $f=E$ on the boundary of $J$. In particular, there exists a point $\bar s\in J$ where $f/E$ reaches its maximum. At this point, one has $f'E-fE'=0$ and $f''E-E''f\leq 0$. Thus
\[
c_0 f^{p_0-2}(\bar s)\leq\frac{f''(\bar s)}{f(\bar s)}\leq \frac{E''(\bar s)}{E(\bar s)}=c E^{p_0-2}(\bar s)\leq c f^{p_0-2}(\bar s),
\]
which contradicts the fact that $c<c_0$. 

Now, assume that $J$ is unbounded, i.e. $J=(r,+\infty)$ with $r\ge s^-$. For all $s\in J$, one has $f''(s)\geq c_0 f^{p_0-1}(s)$. In particular $f$ is convex. Moreover, $f$ is decreasing, since it is convex and $f(s_n)\to 0$ for a sequence $s_n\to \infty$. Indeed, since $\mathfrak{E}_{W}(f,J)<+\infty$, there exists a sequence $s_n\to\infty$ with $f'(s_n)\to 0$ and $W(f(s_n)\sigma(s_n))\to 0$ as $n\to \infty$. Up to reducing $\varepsilon_0$, one can assume that $f(s)\sigma(s)$ stays in a ball centered at $a^+=0$ which does not meet the set $\{W=0\}$ except at $a^+=0$ thus implying $f(s_n)\to 0$. From the inequalities $f'<0$ and $f''\ge c_0f^{p_0-1}$, we deduce that
\[
\left(\frac 12 |f'|^2-\frac{c_0}{p_0}f^{p_0}\right)'=f'(f''-c_0f^{p_0-1})\leq 0.
\]
Since $\frac 12 |f'(s_n)|^2-\frac{c_0}{p_0}f^{p_0}(s_n)\to 0$ as $n\to \infty$, one has also 
\[
\frac 12 |f'|^2-\frac{c_0}{p_0}f^{p_0}\geq 0.
\]
Moreover, by construction, one has $(\frac 12 |E'|^2-\frac{c}{p_0}E^{p_0})'= 0$, and so $\frac 12 |E'|^2-\frac{c}{p_0}E^{p_0}= 0$. Thus
\[
\frac{-f'}{f^{p_0/2}}\geq \sqrt{\frac{2c_0}{p_0}}>\sqrt{\frac{2c}{p_0}}=\frac{-E'}{E^{p_0/2}}.
\]
In particular, we have proved that the following inequalities hold on the interval $J$:
\[
\begin{cases}
(-\log f)'\geq (-\log E)'\text{ and so }-\log f> -\log E&\text{if }p_0=2;\\
(f^{1-p_0/2})'>(E^{1-p_0/2})'\text{ and so }f^{1-p_0/2}>E^{1-p_0/2}&\text{if }p_0>2.
\end{cases}
\]
In both cases, we deduce the inequality $f<E$ which is a contradiction.
\end{proof}

\section{Symmetric case: Alama-Bronsard-Gui connections\label{sectionAlama}}
As in \cite{Alama:1997}, we investigate the existence of solutions $u$ of the system \eqref{Double_connections}, which have the following symmetry property:
\begin{equation}
u\circ\mathcal R_2=\mathcal R_n\circ u,
\end{equation}
where we define $\mathcal R_m:\R^m\to\R^m$ for all integer $m\geq 2$ as the map given by $\mathcal R_m(x_1,x')=(-x_1,x')$ whatever $x_1\in\R$ and $x'\in\R^{m-1}$.

This is made possible only with a symmetry condition on $W$ and on the boundary data $a^\pm\in\Sigma$, namely:
\begin{description}
	\item[(Sym)] $W\circ\mathcal R_n=W$ and $a^+=\mathcal R_n(a^-)$.
\end{description}
With these conditions, the two constants $c^-$ and $c^+$ are actually fixed in the system \ref{Double_connections}: assuming that $z^+$ and $z^-$ are also symmetric, i.e. $z^\pm_1(-t)=-z^\pm_1(t)$, we have $c^-=c^+=0$. In addition to $\mathbf{(Sym)}$, we keep our assumptions $\mathbf{(A1-4)}$. Note that for every heteroclinic connection $z\in\mathcal{Z}$, $t\mapsto \mathcal R_n\circ z (-t)$ is also a heteroclinic connection between $a^-$ and $a^+=\mathcal R_n(a^-)$. Indeed, with our symmetry condition on $W$, one has $\mathcal R_n(z(\mp\infty))=\mathcal R_n(a^\mp)=a^\pm$, and both $t\mapsto z(t)$ and $t\mapsto (\mathcal R_n\circ z) (-t)$ solve the Euler-Lagrange equation $z''=\nabla W(z)$. 

We introduce the following set of symmetric connections between $a^-$ and $a^+$,
\[
\mathcal{S}_{sym}(a^-,a^+):=\{{v}\in\mathcal{S}(a^-,a^+)\;:\; \text{for a.e. }t\in\R,\, {v}_1(-t)=-{v}_1(t)\},
\]
and the set of all heteroclinic connections which minimize the action $\mathfrak{E}_{W}$ over $\mathcal{S}_{sym}(a^-,a^+)$:
\[
\mathcal{Z}_{sym}(a^-,a^+)=\{{z}\in\mathcal{S}_{sym}(a^-,a^+)\;:\;\forall {v}\in\mathcal{S}_{sym}(a^-,a^+),\, \mathfrak{E}_{W}({z})\leq \mathfrak{E}_{W}({v})\}.
\]
We can prove that the minimal energy over symmetric connections is not greater than the minimal energy over all connections:
\begin{lemma}
For every ${v}\in\mathcal{S}(a^-,a^+)$, there exists $v_{sym}\in\mathcal{S}_{sym}(a^-,a^+)$ with $\mathfrak{E}_{W}(v_{sym})\leq \mathfrak{E}_{W}({v})$. In particular, the infimum of $\mathfrak{E}_{W}$ over $\mathcal{S}_{sym}(a^-,a^+)$ is the same as the infimum over $\mathcal{S}(a^-,a^+)$ and thus
\[
\inf\{\mathfrak{E}_{W}({v})\;:\; {v}\in\mathcal{S}_{sym}(a^-,a^+)\}=d_K(u^-,u^+).
\]
\end{lemma}
\begin{proof}
Since ${v}_1(\pm\infty)=a_1^\pm=\pm a_1^+$ and since ${v}_1$ is continuous, there exists $t_0\in\R$ such that ${v}_1(t_0)=0$. Then $\mathfrak{E}_{W}({v})=\mathfrak{E}_{W}({v},(-\infty,t_0])+\mathfrak{E}_{W}({v},[t_0,+\infty))$. If for instance $\mathfrak{E}_{W}({v},(-\infty,t_0])\geq \mathfrak{E}_{W}({v},[t_0,+\infty))$, then $\mathfrak{E}_{W}({v})\geq 2\mathfrak{E}_{W}({v},[t_0,+\infty))=\mathfrak{E}_{W}(v_{sym})$, where $v_{sym}\in\mathcal{S}(a^-,a^+)$ is defined by
\[
v_{sym}(t)=
\begin{cases}
{v}(t_0+t)&\text{if }t>0,\\
\mathcal R_n({v}(t_0-t))&\text{if }t\leq 0.
\end{cases}
\]
Since $v_{sym}$ is clearly symmetric, the lemma is proved.
\end{proof}
The main result of this section is
\begin{theorem}\label{ABGthm}
Assume $\mathbf{(A1-4)}$, $\mathbf{(STI)}$ and $\mathbf{(Sym)}$. Assume moreover that $\mathcal{Z}_{sym}(a^-,a^+)$ has exactly two elements $z^-$ and $z^+$. Then there exists a solution $u\in C^2(\R^2,\R^n)$ to the system  \eqref{Double_connections} (with $c^-=c^+=0$) which globally minimizes the energy $\mathcal{E}$ under the constraints:
\begin{equation}
\label{BC}
\begin{cases}
\int_\R (u(x_1,x_2)-z^+(x_1))^2\d x_1<+\infty&\text{for a.e. }x_2\in\R;\\
\int_\R (u(x_1,x_2)-z^\pm(x_1))^2\d x_1\to 0&\text{as }x_2\to\pm\infty.
\end{cases}
\end{equation}
\end{theorem}
\noindent The proof relies on Theorem \ref{heteroclinic}, applied in the following setting. Consider the metric space
\[
X:=\{{v}=z^++{w}\;:\; {w}\in L^2(\R,\R^n)\text{ s.t. }{w}_1(t)=-{w}_1(-t)\text{ for a.e. }t\in\R\},
\]
endowed with the $L^2$-distance,
\[
d_X({v}_1,{v}_2)=\|{v}_1-{v}_2\|_{L^2(\R,\R^n)}.
\]
Define also a weight function $\mathcal{K}:X\to \R^+\cup\{+\infty\}$ by
\[
\mathcal{K}({v})=
\begin{cases}
\sqrt{\mathfrak{E}_{W}({v})-d_K(a^-,a^+)}&\text{if }{v}\in H^1_{loc}(\R,\R^n),\\
+\infty&\text{otherwise.}
\end{cases}
\]
The following statement is an easy consequence of what proven so far.
\begin{lemma}\label{propX}
The weighted metric space $(X,d_X,\mathcal{K})$ enjoys the following elementary properties:
\begin{itemize}
\item
$(X,d_X)$ is complete and is a length space;
\item
$\mathcal{K}({v})$ vanishes only when ${v}$ is a symmetric heteroclinic connection, i.e.
\[
\mathcal{K}({v})=0\Longleftrightarrow {v}\in\Sigma:=\{z^-,z^+\};
\]
\item
$\mathcal{K}$ is l.s.c. on $(X,d_X)$.
\end{itemize}
\end{lemma}
In this infinite dimensional setting, the difficulty in applying Theorem \ref{heteroclinic} is the compactness assumption $\mathbf{(H3)}$, the only one which does not follow from the preceding lemma. Note that, in order to check this assumption, it is enough to find a subset $F\subset X$ such that $F\cap\{\mathcal K\leq \ell\}$ is compact for every $\ell\in\R$ and 
\begin{equation}
\label{compactness}
\forall\gamma\in AC_{ploc}([0,1],X),\,\forall\ve >0,\, \exists\tilde{\gamma}\in AC_{ploc}([0,1],F),\, \mathfrak{L}_{\mathcal{K}}(\tilde{\gamma})\leq \mathfrak{L}_{\mathcal{K}}(\gamma)+\ve.
\end{equation}
With every curve $\gamma\in AC_{ploc}([0,1],X)$ one can associate a function $\gamma(t,s)$ with two arguments such that for all $t\in[0,1]$, $\gamma(t)=(s\in\R\mapsto \gamma(t,s)) \in X$. In this context, the $\mathfrak{L}_{\mathcal{K}}$-length of $\gamma$ writes 
\[
\mathfrak{L}_{\mathcal{K}}(\gamma)=\int_0^1 \mathcal{K}(\gamma(t))|\dot{\gamma}|(t)\d t=\int_0^1 \mathcal{K}(\gamma(t,\cdot))\|\partial_t\gamma(t,\cdot)\|_{L^2(\R)}\d t.
\]
By Lemma \ref{projection}, replacing $\gamma(t)$ by $P[\gamma(t)]$ for all $t\in[0,1]$ decreases the cost $\mathcal{K}(\gamma(t))$ provided $s_0$ is such that $|\gamma(t,s_0)|<\varepsilon_0$ for every $t$. It is also clear that the metric derivative $|\dot{\gamma}|(t)$ decreases. We will arrive in several steps to prove the existence of a suitable $s_0$ in order to perform this projection. Before, we need the following lemma:
\begin{lemma}\label{stability}
For every $\varepsilon>0$, there exists $\delta>0$ such that for all ${v}\in X$ with $\mathcal{K}({v})<\delta$, one has $\|{v}-z^+\|_{L^\infty(\R)}<\varepsilon$ or $\|{v}-z^-\|_{L^\infty(\R)}<\varepsilon$.
\end{lemma}
\begin{proof}
This is the same as proving that for every sequence $({v}_n)_n\subset X$ such that $\mathcal{K}({v}_n)\to 0$, $\|{v}_n-z^\pm\|_{L^\infty}$ tends to $0$ as $n\to\infty$. Without loss of generality, one can assume that
\[
\|{v}_n-z^\pm\|_{L^\infty}\underset{n\to\infty}{\longrightarrow}\limsup\limits_{n\to\infty}\|{v}_n-z^\pm\|_{L^\infty},
\]
so that we are free to extract a subsequence whenever needed. We need to prove the two claims below:
\begin{claim}\label{centered}
There exist $\delta_1,\varepsilon_1>0$ depending on $W$, $a^-$ and $a^+$ only, such that for all ${v}\in X$ with $\mathcal{K}({v})<\delta_1$ and $t>0$, one has
\[
 |{v}(t)-a^-|=|{v}(-t)-a^+|\geq\varepsilon_1.
\]
\end{claim}
 To prove this, take $\delta_1$ and $\tilde\varepsilon_1$ such that $\delta_1^2+4\tilde\varepsilon_1<2d_K(a^-,a^+)$, then $\rho>0$ such that ${v}\in X$ and $\mathcal{K}({v})<\delta_1$ imply $v(\R)\subset B(a^+,\rho)$, and finally $\varepsilon_1$ such that  for every $x,y\in B(a^+,\rho)$ with $|x-y|<\varepsilon_1$, one has $d_K(x,y)< \tilde\varepsilon_1$ (this is possible thanks to the equivalence between the distance $d_K$ and the Euclidean distance on compact sets).
 
 Now, assume by contradiction that there exist ${v}\in X$ with $\mathcal{K}({v})<\delta_1$ and $t_0>0$ with $|{v}(-t_0)-a^+|<\varepsilon_1$, and thus $d_K({v}(-t_0),a^+)=d_K({v}(t_0),a^-)<\tilde\varepsilon_1$. Hence
\begin{eqnarray*}
\delta_1^2>\mathcal{K}({v})^2&=&\mathfrak{E}_{W}({v})-d_K(a^-,a^+)\geq d_K(a^-,{v}(-t_0))+[d_K({v}(-t_0),{v}(t_0))-d_K(a^-,a^+)]+d_K({v}(t_0),a^+)\\
&\geq&d_K(a^-,a^+)-\tilde\varepsilon_1+[d_K(a^+,a^-)-2\tilde\varepsilon_1-d_K(a^-,a^+)]+d_K(a^-,a^+)-\tilde\varepsilon_1\\
&=&2d_K(a^-,a^+)-4\tilde\varepsilon_1.
\end{eqnarray*}
With our choice of $\delta$ and $\tilde\varepsilon_1$ this is a contradiction.

\begin{claim}\label{claimOtherZero}
There exist $\delta_2,\varepsilon_2>0$ depending on $W$, $a^-$ and $a^+$ only, such that for all ${v}\in X$ with $\mathcal{K}({v})<\delta_2$, $a\in\Sigma\setminus\{a^-,a^+\}$ and $t\in\R$, one has 
$$|{v}(t)-a|\geq\varepsilon_2.$$
\end{claim}
Set $c:=\inf_{a\in\Sigma\setminus\{a^-,a^+\}}  d_K(a^-,a)+d_K(a,a^+)-d_K(a^-,a^+)$. We have $c>0$ by our assumption $(\mathbf{STI})$. Next, take $\delta_2 $ and $\tilde\varepsilon_2$ such that $\delta_2^2+2\tilde\varepsilon_2<c$, then $\rho>0$ such that ${v}\in X$ and $\mathcal{K}({v})<\delta_2$ imply $v(\R)\subset B(a^+,\rho)$, and finally $\varepsilon_2>0$ such that for every $x,y\in B(a^+,\rho)$ with $|x-y|<\varepsilon_2$, one has $d_K(x,y)< \tilde\varepsilon_2$.

 Now, assume by contradiction that there exist ${v}\in X$ with $\mathcal{K}({v})<\delta_2$, and $a_0\in\Sigma\setminus\{a^-,a^+\}$, $t_0\in\R$ with $|{v}(t_0)-a_0|<\varepsilon_2$ and hence $d_K({v}(t_0),a_0)<\tilde\varepsilon_2$.
Thus
\[
\delta_2^2>\mathcal{K}({v})^2\geq d_K(a^-,{v}(t_0))+d_K({v}(t_0),a^+)-d_K(a^-,a^+)\geq d_K(a^-,a_0)+d_K(a_0,a^+)-d_K(a^-,a^+)-2\tilde\varepsilon_2\geq c-2\tilde\varepsilon_2.
\]
Again, with our choice of $\delta_2$ and $\tilde\varepsilon_2$, this is a contradiction.
\medskip

Let us come back to the proof of Lemma \ref{stability}. Up to extracting a subsequence if necessary, one can assume that for all $n$, $\mathcal{K}({v}_n)<\inf\{\delta_1,\delta_2\}$. By the two preceding claims, if $\eta>0$ but $\eta<\varepsilon_1,\varepsilon_2,\varepsilon_0$ (where $\varepsilon_0>0$ is provided by Lemma \ref{projection}), then, for each $n$,
\[
\forall t>0,\, |{v}_n(t)-a^-|=|{v}_n(-t)-a^+|\geq \eta\quad \text{and}\quad \forall a\in\Sigma\setminus\{a^-,a^+\},\, \forall t\in\R,\, |{v}_n(t)-a|\geq\eta.
\]
Now, fix $S>0$ large enough for the following estimate to be satisfied
\[
\inf\big\{W(z)\;:\; |z|\leq C_0,\, \forall a\in\Sigma,\, |z-a|\geq\eta\big\}>\frac{C_0}{S},\quad\text{where}\quad C_0:=\sup_n\; (\|{v}_n\|_{L^\infty}+\mathfrak{E}_{W}({v}_n)).
\]
Since $C_0\geq \mathfrak{E}_{W}({v}_n,[S,2S])\geq \int_S^{2S}W({v}_n)$, this implies that for all $n\geq 0$, there exists $s^+_n\in [S,2S]$ such that $|{v}_n(s^+_n)-a^+|<\eta<\varepsilon_0$. By an application of Lemma \ref{projection}, one has $\mathfrak{E}_{W}(\overline{{v}_n})\leq \mathfrak{E}_{W}({v}_n)$, where $\overline{{v}_n}$ is defined by
\[
\overline{{v}_n}(s)=
\begin{cases}
P[{v}_n](s)&\text{if }s\geq 0,\\
\mathcal R_n(P[{v}_n](-s))&\text{if }s<0,
\end{cases}
\]
where $P[{v}_n]$ is defined as in Lemma \ref{projection} with $s_0=s^+_n$. Thus $(\overline{{v}_n})_n$ is still a minimizing sequence for $\mathfrak{E}_{W}$, i.e. $\mathcal{K}(\overline{{v}_n})\to 0$. Moreover, each $\overline{{v}_n}$ belongs to the following $L^2$-compact and convex subset of $\mathcal{S}(a^-,a^+)$ ($L^2$ compactness comes from the compact injection of $H^1$ into $L^2$ which is true on bounded domains, and from the behavior at infinity, which allows to handle the values outside bounded domains):
\[
\mathcal{C}:=\{{v}\in\mathcal{S}_{sym}(a^-,a^+)\;:\; \mathfrak{E}_{W}(v)\leq C_0 \mbox{ and }\forall t\geq 2S,\, |{v}(t)-a^+|=|{v}(-t)-a^-|\leq E(t)\}.
\]
Up to extraction, one can assume that $(\overline{{v}_n})_n$ converges in $L^2$, and hence, because of the $H^1$ bound, also uniformly on $\R$. In particular, the pointwise limit $\overline{{v}}$ of the sequence $(\overline{{v}_n})_n$ belongs to the set $\mathcal{C}\subset\mathcal{S}_{sym}(a^-,a^+)$. Moreover, as $\mathcal{K}$ is l.s.c. w.r.t. $L^2_{loc}$-convergence, one has 
$\mathcal{K}(\overline{{v}})\leq \liminf \mathcal{K}(\overline{{v}_n})=0$, i.e. $\overline{{v}}$ minimizes $\mathfrak{E}_{W}$. Thus
\begin{equation}\label{dingdong}
\overline{{v}}=z^+\quad\text{or}\quad \overline{{v}}=z^-.
\end{equation}
For the sake of simplicity, let say that $\overline{{v}}=z^+$ so that $z^+$ is the uniform limit of ${v}_n$ ($=\overline{{v}_n}$) on $[-S,S]$. Since this is valid for arbitrary large value of $S$ and by uniqueness of the limit, we have actually proved that ${v}_n$ converges  \emph{locally} uniformly to $z^+$ on $\R$. It remains to prove that the convergence is uniform on the whole space. We use the estimate
\[
\mathfrak{E}_{W}({v}_n)\geq \mathfrak{E}_{W}({v}_n,[-S,S])+\mathfrak{E}_{W}({v}_n,\R\setminus [-S,S]).
\]
Since $\mathcal{K}({v}_n)\to 0$, we get in the limit as $n\to\infty$,
\[
d_K(a^-,a^+)\geq \mathfrak{E}_{W}(z^+,[-S,S])+\limsup\limits_{n\to\infty}\mathfrak{E}_{W}({v}_n,\R\setminus [-S,S]).
\]
In particular, we have that
\[
\lim\limits_{S\to\infty}\limsup\limits_{n\to\infty}\mathfrak{E}_{W}({v}_n,\R\setminus [-S,S])=0.
\]

Now, fix a value $\ve_0$ for which we would like to prove $|v_n-z^+|\leq \ve_0$ on $\R$, and choose $S$ large enough, so that $|z^+(\pm s)-a^\pm|\leq \ve_0/2$ for $s>S$. Then fix $\ve_1$ such that $d_K(x,a^\pm)\leq \ve_1$ implies $|x-a^\pm|\leq \ve_0/2$. By possibly enlarging the value of $S$, also suppose that $\limsup_{n\to\infty}\mathfrak{E}_{W}({v}_n,\R\setminus [-S,S])\leq \ve_1/2$. This means that, for $n$ large enough, we have  $\mathfrak{E}_{W}({v}_n,\R\setminus [-S,S])\leq \ve_1$.
Since the total variation (for the metric $d_K$) of ${v}_n$ out of $[-S,S]$ is controlled by its energy $\mathfrak{E}_{W}({v}_n,\R\setminus [-S,S])$, we deduce $d_K(v_n(\pm s),a^\pm)\leq \ve_1$ for $s>S$. This implies $|v_n(\pm s)-a^\pm|\leq \ve_0/2$ and hence $|v_n(\pm s)-z^+(\pm s)|\leq \ve_0$. It is then enough to choose $n$ large enough to guarantee the same inequality on $[-S,S]$ (using local uniform convergence), and we have proven the desired result. 
\end{proof}
As a consequence of the preceding Lemma, one can prove the following:
\begin{lemma}\label{s0}
Let $\gamma\in AC_{ploc}([0,1],X)$ be a curve parametrized by $\gamma(t)=(s\in\R\mapsto \gamma(t,s))\in X$. Assume that $\gamma(0)=z^-$, $\gamma(1)=z^+$, $\mathfrak{L}_{\mathcal{K}}(\gamma)=\int \mathcal{K}(\gamma(t))|\dot{\gamma}|(t)\d t<\infty$ and that $(t,s)\mapsto\gamma(t,s)$ is continuous. 

Then, for all $\varepsilon_0>0$, there exist $s^-,s^+\in \R$ such that $s^-< s^+$ and for all $t\in [0,1]$ we have $|\gamma(t,s^\pm)-a^\pm|<\varepsilon_0$.
\end{lemma}
\begin{proof}
We prove the existence of $s^+$ ; the existence of $s^-$ is similar. First observe that for all $S\geq 0$, by the Cauchy-Schwarz inequality, one has the estimate
\[
\|\partial_t\gamma(t,\cdot)\|_{L^2(\R)}\geq\|\partial_t\gamma(t,\cdot)\|_{L^2([S,2S])}\geq S^{-1/2}\|\partial_t\gamma(t,\cdot)\|_{L^1([S,2S])}.
\]
This allows to get the following estimate, by Fubini's theorem:
\begin{equation}\label{meanLK1}
\frac 1S\int_{S}^{2S}\int_0^1 \mathcal{K}(\gamma(t))|\partial_t\gamma(t,s)|\d t\d s\leq S^{-1/2}\mathfrak{L}_{\mathcal{K}}(\gamma).\end{equation}
Now, applying Lemma \ref{stability} to $\varepsilon=\varepsilon_0/3$, one gets a constant $\delta>0$ such that $\mathcal{K}({v})\leq\delta$ implies $\|{v}-z^\pm\|_{L^\infty(\R)}<\varepsilon_0/3$. Let us choose $S>0$ large enough so that
\[
S^{-1/2}\mathfrak{L}_{\mathcal{K}}(\gamma)<\frac{\delta\varepsilon_0}{3}\quad \text{and}\quad\forall s>S,\, |z^\pm (s)-a^+|<\frac{\varepsilon_0}{3}.
\]
By \eqref{meanLK1}, there exists $s^+\in [S,2S]$ such that 
\begin{equation}
\label{good_s1}
\int_0^1\mathcal{K}(\gamma(t))|\partial_t\gamma(t,s^+)|\d t< \frac{\delta\varepsilon_0}{3}.
\end{equation}
Let $I\subset [0,1]$ be the set of points $t\in [0,1]$ such that $\mathcal{K}(\gamma(t))>\delta$. Note that $I\subset (0,1)$ since $\mathcal{K}(\gamma(0))=\mathcal{K}(\gamma(1))=0$, and that $I$ is open as $\mathcal{K}\circ\gamma$ is l.s.c. If $t\in [0,1]\setminus I$, one has $\|\gamma(t)-z^\pm\|_{L^\infty(\R)}<\varepsilon_0/3$ and then
\[
|\gamma(t,s^+)-a^+|\leq |\gamma(t,s^+)-z^\pm(s^+)|+|z^\pm(s^+)-a^+|< \frac{2\varepsilon_0}{3}.
\]
For the points in $I$, note that, by \eqref{good_s1}, one can estimate the total variation of $t\mapsto\gamma(t,s^+)$ on $I$ as follows
\[
\int_I|\partial_t\gamma(t,s^+)|\d t\leq\frac{\varepsilon_0}{3}.
\]
Since $t\mapsto \gamma(t,s^+)$ is continuous and $|\gamma(t,s^+)-a^+|< \frac{2\varepsilon_0}{3}$ on the boundary of $I$, we have proven $|\gamma(t,s^+)-a^+|<\varepsilon_0$ which was the claim.
\end{proof}

In order to be able to use the previous Lemma, we need to establish the following useful regularization property.
\begin{lemma}\label{smoothing}
Given a curve $\gamma\in AC_{ploc}([0,1],X)$ with $\gamma(0)=z^-$, $\gamma(1)=z^+$, $\mathfrak{L}_{\mathcal{K}}(\gamma)=\int \mathcal{K}(\gamma(t))|\dot{\gamma}|(t)\d t<\infty$ and $\ve>0$, there exists a curve $\tilde\gamma\in AC_{ploc}([0,1],X)$ parametrized by $\tilde\gamma(t)=(s\in\R\mapsto\tilde \gamma(t,s))\in X$ such that $\tilde\gamma(\cdot,\cdot)$ is continuous on $\R^2$, $\tilde\gamma(0)=z^-$, $\tilde\gamma(1)=z^+$ and $\mathfrak{L}_{\mathcal{K}}(\tilde\gamma)<\mathfrak{L}_{\mathcal{K}}(\gamma)+\ve$.
\end{lemma}
\begin{proof}
First, let us prove that there exists a sequence of curves $\gamma_n$ with finite length in $X$ (i.e. $\int|\dot{\gamma_n}|(t)\d t<\infty$) and $\limsup_n \mathfrak{L}_{\mathcal{K}}(\gamma_n)\leq \mathfrak{L}_{\mathcal{K}}(\gamma)$. Of course, if $\inf_t \mathcal{K}(\gamma(t))>0$ then one can simply take $\gamma_n=\gamma$, because then $\int|\dot{\gamma}|(t)\d t<\infty$. Also, we can assume (up to removing cycles where $\gamma$ takes many times the value $z^+$ or $z^-$) that $\mathcal{K}(\gamma(t))>0$ for every $t\neq 0,1$ since $\mathcal K$ only vanishes at the two boundary data $z^\pm$.  Using the semicontinuity of $\mathcal K$ this means that, if 
$\inf_t \mathcal{K}(\gamma(t))=0$, then there exists either a sequence $t_n^-\to 0$ or a sequence $t_n^+\to 1$ (or both) with $\mathcal{K}(\gamma(t_n^\pm))\to 0$. Define the curve $\gamma_n$ by replacing $\gamma$ on $[0,t_n^-]$ with a constant-speed segment joining $\gamma(0)$ and $\gamma(t_n^-)$, i.e. $\gamma_n(t)=(1-t/t_n^-)\gamma(0)+t/t_n^-\gamma(t_n^-)$, and do a similar construction for $t_n^+\to 1$ (only for those among these two sequences which are actually present). Note that, by continuity of the curve $\gamma$, we have $d(\gamma(t_n^-),\gamma(0))\to 0$ (same for $t_n^+\to 1$). This construction provides a curve with finite length. The cost $\mathfrak{L}_{\mathcal{K}}(\gamma_n)$ can only increase, compared to $\mathfrak{L}_{\mathcal{K}}(\gamma)$, in what concerns the intervals $[0,t_n^-]$ and $[t_n^+,1]$. Note that the functional $\mathcal{W}=\mathcal{K}^2/2$ is not convex (because $W$ is not convex), but is $\lambda$-convex for a negative $\lambda$ given by the lower bound of $\nabla^2W$. This means that for some constant $C>0$,
\begin{equation}\label{estimateWconv}
\mathcal W(\gamma_n(t))\leq \left(1-\frac{t}{t_n^-}\right) \mathcal W(\gamma(0))+\frac{t}{t_n^-} \mathcal W(\gamma(t_n^-))+Cd(\gamma(0),\gamma(t_n^-))^2\quad \text{on }[0,t_n^-]
\end{equation}
and, using $\mathcal W(\gamma(0))=0$ and the subadditivity of the square root, one gets $\mathcal K(\gamma_n(t))\leq \mathcal K(\gamma(t_n^-))+\sqrt{C}\, d(\gamma(0),\gamma(t_n^-))$. Using $|\dot{\gamma_n}|(t)=d(\gamma(0),\gamma(t_n^-))/t_n^-$ on $[0,t_n^-]$ and performing the same estimates on $[t_n^+,1]$, one gets 
$$
\mathfrak{L}_{\mathcal{K}}(\gamma_n)\leq \mathfrak{L}_{\mathcal{K}}(\gamma)+\big[\mathcal K(\gamma(t_n^-))+\sqrt{C}\, d(\gamma(0),\gamma(t_n^-))\big]d(\gamma(0),\gamma(t_n^-))+\big[\mathcal K(\gamma(t_n^+))+\sqrt{C}\, d(\gamma(1),\gamma(t_n^+))\big]d(\gamma(1),\gamma(t_n^+)).$$
Since we have $\mathcal K(\gamma(t_n^-)),\,\mathcal K(\gamma(t_n^+)),\, d(\gamma(0),\gamma(t_n^-)),\, d(\gamma(1),\gamma(t_n^+))\to 0$, we obtain
$\limsup_n \mathfrak{L}_{\mathcal{K}}(\gamma_n)\leq \mathfrak{L}_{\mathcal{K}}(\gamma)$.

Up to replacing $\gamma$ with one of these curves $\gamma_n$, we can now assume that $\gamma$ has finite length. Then we apply a convolution, i.e. we replace each $\gamma(t,\cdot)$ with $\gamma(t,\cdot)*\rho$, where $\rho$ is a standard mollifier with unit mass and support contained in $[-\delta,\delta]$ with $\delta>0$. We call $\gamma_*$ the new curve we obtain in this way. The convolution reduces the metric derivative in $L^2$, but could increase the value of $\mathcal W$, and also change the initial and final data. Let us look at how much $\mathcal W$ can increase. We claim that we have, for every function ${v}\in H^1_{loc}(\R,\R^n)$, the following inequalities
\begin{equation}\label{estimateWrho}
\mathcal W({v}*\rho)\leq \mathcal W({v})+8\delta^2|\lambda|\int\frac12 |{v}'|^2\leq  \mathcal W({v})+8\delta^2|\lambda|(\mathcal W({v})+d_K(a^-,a^+)).
\end{equation}
Once we have this inequality, it is clear that we have $\mathfrak{L}_{\mathcal{K}}(\gamma_*)\leq  \sqrt{1+8\delta^2|\lambda|}\, \mathfrak{L}_{\mathcal{K}}(\gamma)+\delta\sqrt{8|\lambda|c}\int|\dot{\gamma}|(t)\d t$ and that this last quantity can be made as close to $\mathfrak{L}_{\mathcal{K}}(\gamma)$ as we want by choosing $\delta$ small. Yet, we still need to modify $\gamma_*$ since $\gamma_*(0)$ and $\gamma_*(1)$ are not equal to $z^\pm$ but to $z^\pm*\rho$. Note that, using \eqref{estimateWrho} and $\mathcal W(z^\pm)=0$, we have $\mathcal W(z^\pm*\rho)\leq 8c\delta^2$.
 In order to modify the initial and final data, we consider a curve connecting $z^\pm$ to $z^\pm*\rho$ via a constant-speed segment and, using the same estimate as we did above in \eqref{estimateWconv}, this connection has an $\mathfrak{L}_{\mathcal{K}}$ length which is at most equal to 
$$d(z^\pm,z^\pm*\rho)(\mathcal K(z^\pm*\rho)+\sqrt{C}d(z^\pm,z^\pm*\rho))\leq C'd(z^\pm,z^\pm*\rho)(\delta+d(z^\pm,z^\pm*\rho)),$$
a quantity which tends to $0$ as $\delta\to 0$. We can then build, and reparametrize on $[0,1]$, a curve which uses this connection from $z^-$ to $z^-*\rho$, then uses $\gamma_*$ from $z^-*\rho$ to $z^+*\rho$,  and then the connection from $z^+*\rho$ to $z^+$.

We are just left with proving \eqref{estimateWrho}. Note that the $H^1$ part of the energy $\mathcal W$ decreases by convolution, so we just look at the integral of $W$. Writing for simplicity $\bar {v}$ instead of ${v}*\rho$, we have
$$W({v}(t))\geq W(\bar {v}(s))+\nabla W(\bar {v}(s))\cdot ({v}(t)-\bar {v}(s))-\frac{|\lambda|}{2}|{v}(t)-\bar {v}(s)|^2.$$
We multiply times $\rho(t-s)$ and integrate in $\d s\d t$, thus getting
\begin{eqnarray*}
\int_\R  W({v}(t))\d t&=&\int_\R  \int_\R W({v}(t))\rho(t-s) \d s\d t\\
			&\geq&\int_\R \int_\R \left(W(\bar {v}(s))+\nabla W(\bar {v}(s))\cdot ({v}(t)-\bar {v}(s))-\frac{|\lambda|}{2}|{v}(t)-\bar {v}(s)|^2\right)\rho(t-s)\d s\d t\\
			&=&\int_\R  W(\bar {v}(s))\d s-\int_\R \int_\R \frac{|\lambda|}{2}|{v}(t)-\bar {v}(s)|^2\rho(t-s)\d s\d t,
\end{eqnarray*}
where the term with $\nabla W(\bar {v}(s))$ has disappeared since $\int  ({v}(t)-\bar {v}(s))\rho(t-s) \d t=0$, and the term with $W(\bar {v}(s))$ has been first integrated w.r.t. $t$. Then consider that $\rho(t-s)>0$ implies $|t-s|<\delta$, and we can write
$$|{v}(t)-\bar {v}(s)|^2\leq\left| \int_\R |{v}(t)-{v}(t')|\rho(s-t')\d t'\right|^2\leq \int_\R |{v}(t)-{v}(t')|^2\rho(s-t')\d t'.$$
From $|t-s|<\delta$ and $|s-t'|<\delta$ we deduce $|t-t'|<2\delta$, and hence			
$$ |{v}(t)-{v}(t')|^2\leq |t-t'|\int_{[t,t']}|{v}'|^2\d s'\leq 2\delta \int_{t-2\delta}^{t+2\delta}|{v}'|^2\d s'.$$
This allows to obtain, for $|t-s|<\delta$, the bound
$$|{v}(t)-\bar {v}(s)|^2\leq 2\delta \int_{t-2\delta}^{t+2\delta}|{v}'|^2\d s'.$$
We then obtain
\begin{align*}
\int_\R  W({v}(t))\d t&\geq \int_\R  W(\bar {v}(s))\d s-|\lambda|\delta \int_\R \int_\R \rho(t-s) \left(\int_{t-2\delta}^{t+2\delta}|{v}'(s')|^2\d s'\right)\d t\d s\\
&=\int_\R W(\bar {v}(s))\d s-|\lambda|\delta \int_\R \int_{-2\delta}^{2\delta}|{v}'(t+s')|^2\d s'\d t\\
&=\int_\R  W(\bar {v}(s))\d s-4\delta^2|\lambda|\int_\R |{v}'(s')|^2\d s'.
\end{align*}
We conclude by observing that $\mathcal W({v})=\int (\frac 12 |{v}'|^2+W({v})) -d_K(a^+,a^-)$, so that $\int\frac 12 |{v}'|^2\leq \mathcal W({v})+d_K(a^+,a^-)$.
\end{proof}

\begin{proof}[Proof of Theorem \ref{ABGthm}]

\textsc{$\bullet$ Existence of a $\mathfrak{E}_{\mathcal{W}}$-minimizing curve between $z^-$ and $z^+$.} Theorem \ref{heteroclinic}, applied to our metric space $(X,d_X)$ endowed with the potential $\mathcal{W}=\mathcal{K}^2/2$ and the two zeros $z^-$ and $z^+$, provides a curve $\gamma\in\Lip(\R,X)$ such that
\[
\forall\sigma\in AC_{ploc}(\R,X),\, \sigma:z^-\mapsto z^+\Longrightarrow \mathfrak{E}_{\mathcal{W}}(\gamma)\leq \mathfrak{E}_{\mathcal{W}}(\sigma).
\]
Let us check the only non trivial assumption in applying Theorem \ref{heteroclinic}, that is $\mathbf{(H3)}$. Actually, it is enough to prove \eqref{compactness}. Given a curve $\gamma:t\mapsto \gamma(t,\cdot)\in X$ in $AC_{ploc}([0,1],X)$, that can be assumed to have finite $\mathfrak{L}_{\mathcal{K}}$-length, we first apply Lemma \ref{smoothing} in order to replace it with a new curve which is a continuous function of its two arguments and with a cost $\mathfrak{L}_{\mathcal{K}}$ which is at most slightly larger than the original one. We still denote by $\gamma$ this new curve. Then Lemma \ref{s0} provides two instants $s^-<s^+$ with $|\gamma(t,s^\pm)-a^\pm|<\varepsilon_0$ for a.e. $t$. Moreover, by Lemma \ref{projection} and an obvious variant of the same lemma backward in time (where $a^+$ is replaced by $a^-$), $\mathcal{K}(\gamma(t))$ is reduced when projecting $\gamma(t)$ onto the funnel
\[
\mathcal{C}:=\{{v}\in X\;:\; \text{for a.e. $s\in\R$ s.t. }\pm(s-s^\pm)\geq 0,\, |{v}(s)-a^\pm|\leq E(s)\}.
\]
This projection $\tilde{\gamma}(t)=\tilde{\gamma}(t,\cdot)$ writes as follows: for all $t\in[0,1]$ and $s\in\R$,
\[
\tilde{\gamma}(t,s)=
\begin{dcases}
a^++E(s)\, \frac{\gamma(t,s)-a^+}{|\gamma(t,s)-a^+|}&\text{if }s>s^+\text{ and }|\gamma(t,s)-a^+|>E(s),\\
a^-+E(s)\, \frac{\gamma(t,s)-a^-}{|\gamma(t,s)-a^-|}&\text{if }s<s^-\text{ and }|\gamma(t,s)-a^-|>E(s),\\
\gamma(t,s)&\text{otherwise.}
\end{dcases}
\]
It is also clear that projecting on convex sets in the target space (one for each $s$) reduces the $L^2$-metric derivative. Thus one has also $|\dot{\tilde{\gamma}}|(t)\leq |\dot{\gamma}|(t)$ for a.e. $t\in[0,1]$, and we have proved
\[
\mathfrak{L}_{\mathcal{K}}(\tilde{\gamma})\leq \mathfrak{L}_{\mathcal{K}}(\gamma).
\]
Since $\mathcal{C}$ is a subset of $X$ such that its intersections with the sublevel sets of $\mathcal K$ are compact, $\mathbf{(H3)}$ is satisfied.
\medskip

\textsc{$\bullet$ Improvement in the boundary conditions and $L^\infty$-bound.}  It is clear, by \eqref{energy_length} and by construction of $\gamma$ as a minimizer of $\mathfrak{L}_{\mathcal{K}}$, that $u\in L^2_{loc}(\R^2,\R^n)$, given by ${u}(x_1,x_2)=\gamma(x_2)(x_1)$ for a.e. $x_1$, $x_2$, minimizes $\mathcal{E}$ under the constraints detailed in Theorem \ref{ABGthm}, i.e. \eqref{BC}. 

So far, we only know that ${u}$ satisfies the boundary conditions \eqref{BC} in a very weak sense. By construction and by use of the equation, one can improve it a little bit. First of all, one has ${u}(\cdot,x_2)=\gamma(x_2)\in \mathcal{C}$ for a.e. $x_2$ which implies convergence when $x_1\to\pm\infty$, uniform in  $x_2$ (with a rate given by the function $E$). Concerning the limit when $x_2\to\pm\infty$, up to now, we have only proved $L^2$-convergence of ${u}(\cdot,x_2)$ onto $z^\pm$. Yet, using the result presented in the Appendix (Corollary \ref{Wbounded}), we can infer that $\|\partial_2 {u}(\cdot,x_2)\|_{L^2}$ and $\mathcal W({u}(\cdot,x_2))$ are uniformly bounded in $x_2$. In particular, the derivative w.r.t. $x_2$ of ${u}(\cdot,x_2)$ is bounded in $L^2(\R)$, which turns the $L^2$ convergence into unform convergence. Moreover, from 
$$d_K({u}(x_1,x_2),a^\pm)\leq \mathfrak{E}_{W}({u}(\cdot,x_2))= d_K(a^-,a^+)+2\mathcal W({u}(\cdot,x_2))\leq C,$$
we also infer that ${u}(x_1,x_2)$ is bounded independently of $(x_1,x_2)$. 
\medskip

\textsc{$\bullet$ Euler-Lagrange equation and improvement in the regularity.} For the moment, it is not clear that ${u}\in L^\infty\cap H^1_{loc}(\R^2,\R^n)$ solves the Euler-Lagrange equation $-\Delta u+\nabla W(u)=0$ associated to the energy $\mathcal{E}$, due to the symmetry constraint $u_1(-x_1,x_2)=-u_1(x_1,x_2)$. One can only say that for all $\psi\in \mathcal{C}_c^2(\R^2,\R^n)$ such that $\psi_1(-x_1,x_2)=-\psi_1(x_1,x_2)$, one has
\[
\int_{\R^2}-{u}\cdot\Delta\psi + \nabla W({u})\cdot\psi=0.
\]
Given any function $\varphi\in \mathcal{C}_c^2(\R^2,\R^n)$, define its projection on the symmetry constraint by
\[
P\varphi=\left(\frac{\varphi_1(\cdot,\cdot)-\varphi_1(-\cdot,\cdot)}{2},\varphi_2,\dots,\varphi_n\right).
\]
Thus, for all $\varphi\in \mathcal{C}_c^2(\R^2,\R^n)$, one has
\begin{align*}
0=\int_{\R^2}-{u}\cdot\Delta P\varphi + \nabla W({u})\cdot P\varphi=\int_{\R^2}-P{u}\cdot\Delta \varphi + P\nabla W({u})\cdot \varphi,
\end{align*}
which means that ${u}$ is a distributional solution of the equation $P(-\Delta {u}+\nabla W({u}))=0$, that is
\[
\begin{cases}
-\Delta {u}_1+D_1 W({u})\text{ is even,}\\
-\Delta {u}_2+D_2 W({u})=0.
\end{cases}
\]
Yet, by our symmetry assumptions on $u$ and $W$, it is also clear that $-\Delta {u}_1+D_1 W({u})$ is odd (in the distributional sense, i.e. in the duality with smooth functions). Thus one has proved 
$$\Delta {u}=\nabla W({u}).$$
This allows to obtain higher regularity for ${u}$: the right hand-side being bounded, ${u}$ will be locally $W^{2,p}$ for every $p$, hence $C^{1,\alpha}$. By a bootstrap argument, for $W\in C^2$, we get ${u}\in C^{2,\alpha}$ (in case $W\in C^\infty$, we also get ${u}\in C^\infty$).
\end{proof}

\section{Asymmetric case: Schatzman connections\label{sectionSchatzman}} In \cite{Schatzman:2002}, Michelle Schatzman generalized the above existence result to non symmetric potentials and non symmetric solutions of \eqref{Double_connections}. Thus the constants $c^-$ and $c^+$, appearing in \eqref{Double_connections}, are now unknown of the problem. We remind that the set of all minimizing heteroclinic connections is denoted by $\mathcal{Z}$, and that $\mathcal{Z}$ is translation invariant. In the Alama-Bronsard-Gui situation, this translation invariance was ruled out by the symmetry condition. Here, whatever $z\in\mathcal{Z}$, we define the set $\mathcal{C}(z)$ composed by $z$ and all its translations:
\[
\mathcal{C}(z)=\{z(\cdot-m)\;:\; m\in\R\}.
\]
We will need the following assumption:
\begin{description}
\item[(A5)]
the set $\{\mathcal{C}(z)\;:\; z\in\mathcal{Z}\}$ has exactly two elements;
\end{description}
these two elements $\mathcal{C}(z^-)$ and $\mathcal{C}(z^+)$ correspond to minimizers $z^+$, $z^-\in\mathcal{Z}$ of $\mathfrak{E}_{W}$ which cannot be deduced by translation one from another. Since any heteroclinic connection $z\in\mathcal{Z}$ is solution of the Euler-Lagrange equation,
\[
-z''+\nabla W(z)=0,
\]
we know that $z'$ is in the kernel of the linearized operator $A(z)$, defined on $L^2(\R,\R^n)$ by
\begin{equation}
\label{definitionSpectralA}
D(A(z))=H^2(\R,\R^n),\quad A(z){v}=-{v}''+(\nabla^2 W(z){v}^T)^T.
\end{equation}
It is clear that $A(z)$ is self-adjoint and, by the second order optimality conditions on $z$ (as a minimizer of $\mathfrak{E}_{W}$), that $A(z)$ is nonnegative. Indeed, for every $v\in L^2(\R,\R^n)$, one has
\[
(A(z)v\,;\, v)_{L^2}=\int_\R\left(|v'(s)|^2+D^2W(z(s))(v(s),v(s))\right)\d s,
\]
which is nonnegative since it is nothing but twice the second order variation of $\mathfrak{E}_{W}$ around the minimizer $z$ and under the perturbation $v$. In particular, the spectrum $\sigma(A(z))$ of $A(z)$ is included in $[0,+\infty)$. We will need the following spectral assumption:
\begin{description}
	\item[(Spec)]
	 when $z=z^+$ or $z=z^-$, the kernel of $A(z)$ is one-dimensional and $0$ is isolated in $\sigma(A(z))$.
\end{description}
In other words we assume two things: i) $0$ is in the descrete spectrum, which is the case in particular if the symmetric matrices $D^2W(a^\pm)$ are positive definite (this was assumed in \cite{Schatzman:2002}) ; this means that the essential spectrum of $A(z)$ is included in $[c,+\infty)$ for some $c>0$ ; ii) the multiplicity of the eigenvalue $0$ is exactly $1$ ; the eigenspace $\mathrm{ker}(A(z))$ is thus generated by $z'$.

This spectral condition (more precisely, the nondegeneracy of $D^2W(a^\pm)$ and the fact that $0$ is an eigenvalue of multiplicity one) is the key assumption of \cite{Schatzman:2002} in order to overcome the lack of compactness due to the translation invariance, and it is proved to be generic \cite[Theorem 4.3., Remark 4.4.]{Schatzman:2002}. 

By the max-min characterization of the descrete spectrum, $\mathbf{(Spec)}$ is equivalent to the following explicit estimate, which is what we actually use in the proof:
\begin{equation}
\label{maxMin}
\exists c_0>0,\, \forall v\in L^2(\R,\R^n),\,\forall z\in \{z^-,z^+\},\quad (v\,;\, z')_{L^2}=0\Longrightarrow (A(z)v\,;\, v)_{L^2}\ge c_0 \|v\|_{L^2}.
\end{equation}
The main result of this section is the following theorem (our assumptions are slightly more general but very close to that of \cite{Schatzman:2002}):
\begin{theorem}\label{Schatzman_thm}
Under Assumptions $\mathbf{(A1-5)}$, $(\mathbf{STI})$ and $\mathbf{(Spec)}$, there exists a solution ${u}\in C^2(\R^2,\R^n)$ to the system \eqref{Double_connections} (where $c^-$, $c^+$ are free parameters) which globally minimizes the energy $\mathcal{E}$ under the constraints:
\begin{equation}
\label{BC_Schatzman}
\begin{dcases}
\int_\R ({u}(x_1,x_2)-z^+(x_1))^2\d x_1<+\infty\quad \mathrm{for\ a.e.\ }x_2\in\R;\\
\inf\left\{\int_\R ({u}(x_1,x_2)-z^\pm(x_1-c))^2\d x_1\;:\; c\in\R\right\}\to 0\quad\mathrm{when\ }x_2\to\pm\infty.
\end{dcases}
\end{equation}
\end{theorem}
The rest of this section is dedicated to the proof of the above theorem, and all its assumptions are thus assumed. We will apply Theorem \ref{heteroclinic} in the following setting. In the functional space $z^++L^2(\R,\R^n)$, consider the equivalence relation
\[
v_1\sim v_2\Longleftrightarrow (v_1=v_2)\text{ or }(v_1,v_2\in\mathcal{C}(z^+))\text{ or }(v_1,v_2\in\mathcal{C}(z^-)).
\]
 We consider the metric space $X$ composed of all equivalence classes in $z^++L^2(\R,\R^n)$, endowed with the metric
\[
d_X(v_1,v_2)=\min\bigg\{d_{L^2}(v_1,v_2)\, ;\, d_{L^2}(v_1,\mathcal{C}(z^-))+d_{L^2}(v_2,\mathcal{C}(z^-))\,;\,d_{L^2}(v_1,\mathcal{C}(z^+))+d_{L^2}(v_2,\mathcal{C}(z^+))\bigg\},
\]
where $d_{L^2}$ stands for the $L^2$-distance, $d_{L^2}(v_1,v_2)=\|v_1-v_2\|_{L^2(\R,\R^n)}$. Note that we do not identify all functions with their translations, which is convenient because this means that far from $\mathcal{C}(z^\pm)$ we are exactly considering the $L^2$-metric. We only identify $z^+$ with its own translations, and the same for $z^-$. Since $\mathfrak{E}_{W}$ is translation invariant, the following definition of the weight function $\mathcal{K}:X\to \R^+\cup\{+\infty\}$ makes sense: for very $[{v}]\in X$ with ${v}\in z^++L^2(\R,\R^n)$,
\[
\mathcal{K}([{v}])=
\begin{cases}
\sqrt{\mathfrak{E}_{W}({v})-d_K(a^-,a^+)}&\text{if }{v}\in H^1_{loc}(\R,\R^n),\\
+\infty&\text{otherwise.}
\end{cases}
\]
In the sequel, for the sake of simplicity of notations, we will frequently omit the distinction between ${v}$ and $[{v}]$. The proof of the following statement involves rather standard tools and corresponds to Lemma \ref{propX}.
\begin{lemma}\label{lemmaPropX}
The weighted metric space $(X,d_X,\mathcal{K})$ enjoys the following elementary properties:
\begin{itemize}
\item
$(X,d_X)$ is complete and is a length space;
\item
$\mathcal{K}({v})$ vanishes only when ${v}$ is a symmetric heteroclinic connection, i.e.
\[
\mathcal{K}({v})=0\Longleftrightarrow {v}\in\Sigma:=\{\mathcal{C}(z^-),\mathcal{C}(z^+)\};
\]
\item
$\mathcal{K}$ is l.s.c. on $(X,d_X)$;
\item
the metric derivative in $X$ coincides with the metric derivative in $L^2$; more precisely, for each curve $\gamma\in AC_{ploc}(I,X)$ parametrized by $\gamma(t)=(s\mapsto \gamma(t,s))$, one has
\[
|\dot{\gamma}|(t)=\|\partial_t\gamma(t,\cdot)\|_{L^2(\R)}\quad\text{for a.e. $t$ such that } \gamma(t)\notin\Sigma.
\]
\end{itemize}
\end{lemma}
We need a first estimate which, given a curve $\gamma$ on $X$, gives the best way of reducing the $\mathfrak{L}_{\mathcal{K}}$-length by translating each of the $\gamma(t)$:
\begin{lemma}\label{translation}
Let $\gamma\in AC_{loc}(I,X)$ be a curve parametrized by $\gamma(t)=(s\mapsto \gamma(t,s))$ such that $\mathfrak{L}_{\mathcal{K}}(\gamma)<\infty$ and for a.e. $t\in I$, $\gamma(t)\notin \Sigma$. Let $m\in W^{1,1}_{loc}(I,\R)$ be defined via $m(0)=0$ and
\[
m'(t)=\frac{(\partial_t\gamma(t,\cdot),\partial_s\gamma(t,\cdot))_{L^2(\R)}}{\|\partial_s\gamma(t,\cdot)\|_{L^2(\R)}}\quad\text{a.e.}
\]
Then one has $\mathfrak{L}_{\mathcal{K}}(\tilde{\gamma})\leq \mathfrak{L}_{\mathcal{K}}(\gamma)$, where $\tilde{\gamma}$ is defined by $\tilde{\gamma}(t,s)=\gamma(t,s-m(t))$ for all $t\in I$ and $s\in\R$.
\end{lemma}
\begin{remark}
The curve $\gamma$ is such that $\tilde{\gamma}=\gamma$ if and only if $(\partial_t\gamma,\partial_s\gamma)_{L^2(\R)}=0$ a.e.
\end{remark}
\begin{proof}
Let us first justify that $m$ is well defined (i.e. an antiderivative of the right-hand side exists). Since $\gamma\in AC_{loc}(I,X)$, one has $\partial_t\gamma(t,\cdot)\in L^2(\R,\R^n)$ for a.e. $t\in I$. Moreover, as $\mathfrak{L}_{\mathcal{K}}(\gamma)=\int \mathcal{K}(\gamma(t))|\dot{\gamma}|(t)\d t<\infty$, and with our convention $+\infty\times 0=+\infty$, we know that $\mathcal{K}(\gamma(t))<\infty$ and, in particular, $\partial_s\gamma(t,\cdot)\in L^2(\R,\R^n)$ for a.e. $t\in I$. Due to the constraint $\gamma(t,\cdot)-z^+\in L^2(\R,\R^n)$, we also know that $\|\partial_s\gamma(t,\cdot)\|_{L^2}>0$ for a.e. $t\in I$. Moreover, one has the estimate 
\[
\frac{(\partial_t\gamma(t,\cdot),\partial_s\gamma(t,\cdot))_{L^2}}{\|\partial_s\gamma(t,\cdot)\|_{L^2}}\leq \|\partial_t\gamma(t,\cdot)\|_{L^2}= |\dot{\gamma}|(t)\in L^1_{loc}(I),
\]
so that $t\mapsto m(t)$ is well defined and unique on $I$.

Assume now that $m$ is an arbitrary function in  $W^{1,1}_{loc}(I,\R)$. It is clear that one reduces $\mathfrak{L}_{\mathcal{K}}(\gamma)$ by replacing $\gamma$ by $\tilde{\gamma}$ if $m$ is chosen in such a way that $|\dot{\tilde{\gamma}}|(t)$ is minimal since $\mathcal{K}$ is invariant by translation. Yet,
\[
|\dot{\tilde{\gamma}}|(t)=\|\partial_t\gamma(t,s-m(t))-\partial_s\gamma(t,s-m(t))m'(t)\|_{L^2},
\]
which is minimal exactly when $m'(t)$ is given by the claimed formula.
\end{proof}
Due to the translation invariance of $\mathcal{K}$ and without symmetry conditions, we need new tools to avoid oscillations. For all ${v}\in X$, we introduce the set $M({v})$ of optimal translation parameters $m$ in projecting ${v}$ onto $\mathcal{C}(z^-)\cup\mathcal{C}(z^+)=\{z^\pm(\cdot-m)\;:\;m\in \R\}$:
\[
M({v}):=\left\{m\in\R\;:\; \|{v}-z^\pm(\cdot-m)\|_{L^2}=d_{L^2}({v},\mathcal{C}(z^+)\cup\mathcal{C}(z^-))\right\},
\]
where $z^\pm$ is either equal to $z^+$ if $\|{v}-z^+(\cdot-m)\|_{L^2}<\|{v}-z^-(\cdot-m)\|_{L^2}$ or equal to $z^-$ otherwise. The fact that $M({v})$ is not empty follows from the lower semicontinuity of $m\mapsto\|{v}-z^\pm(\cdot-m)\|_{L^2}$ and the following coercivity property:
\[
\|{v}-z^\pm(\cdot-m)\|_{L^2}\geq\|z^\pm-z^\pm(\cdot-m)\|_{L^2}-\|z^\pm-{v}\|_{L^2}\underset{|m|\to\infty}{\sim}|a^+-a^-|\sqrt{|m|}.
\]
Note that the preceding estimate also shows that $M(v)$ is uniformly bounded over $d_{L^2}$-bounded subsets of $X$:
\begin{equation}
\label{uniformBound}
\forall R>0,\quad \sup\{|m|\;:\; m\in M({v})\text{ with }{v}\in X\text{ s.t. }d_{L^2}({v},\{z^-,z^+\})\leq R\}<+\infty.
\end{equation}
On $d_X$-bounded subsets of $X$, we know at least that the diameter of $M({v})$ is bounded. More precisely, if $A\subset X$ is  $d_X$-bounded, we cannot say that the diameter of $\bigcup_{v\in A}M(v)$ is finite, but we can say that $\sup\{\mathrm{diam}(M(v))\,:\,v\in A\}$ is finite, i.e.
\begin{equation}
\label{diameter}
\forall R>0,\quad\sup\left\{|m_1-m_2|\;:\; m_1,m_2\in M({v})\text{, where }{v}\in X\text{ with }d_X({v},\Sigma)\leq R\right\}<\infty.
\end{equation}
Indeed, one has the estimate
\[
\|z^\pm(\cdot-m_1)-z^\pm(\cdot-m_2)\|_{L^2}\leq\|z^\pm(\cdot-m_1)-{v}\|_{L^2}+\|{v}-z^\pm(\cdot-m_2)\|_{L^2}\leq 2 R,
\]
thus yielding a bound on $|m_1-m_2|$ since the first term is equivalent to $|a^+-a^-||m_1-m_2|^{1/2}$ as $|m_1-m_2|\to\infty$. We need the following lemma:
\begin{lemma}\label{stability_general}
For all $R>0$ and $\varepsilon>0$, there exist $\delta>0$ and $\alpha>0$ with the following properties:
\begin{enumerate}
\item
for every ${v}\in X$ with $d_X({v},\Sigma)\leq\delta$, $M({v})$ is reduced to a single point $m({v})$ and the map ${v}\mapsto m({v})$ is Lipschitz continuous in $L^2$; namely, there exists a constant $C>0$ such that for every ${v}_1,\, {v}_2\in X$ with $d_X({v}_1,\Sigma)\le\delta$ and $d_X({v}_2,\Sigma)\leq\delta$, one has $|m({v}_1)-m({v}_2)|\le C\|{v}_1-{v}_2\|_{L^2}$;
\item 
for every ${v}\in X$ with $\mathcal{K}({v})\leq\delta$ and $d_X({v},\Sigma)\leq R$, one has
\begin{enumerate}
\item
$\|{v}-z\|_{L^\infty}\leq\varepsilon$\ ,
\item
$\mathcal{K}({v})\geq\alpha \|{v}-z\|_{H^1(\R)}$\ ,
\end{enumerate}
\end{enumerate}
where $z\in \mathcal{C}(z^-)\cup\mathcal{C}(z^+)$ is such that $\|{v}-z\|_{L^2}$ is minimal, i.e. $z=z^\pm(\cdot-m)$ with $m\in M({v})$.
\end{lemma}
\begin{remark}\label{Ksmall}
Imposing $d_X({v},\Sigma)\leq R$ is just a way of restricting to bounded subsets of $X$. Note that 
\[
d_X({v},\Sigma)=d_{L^2}({v},\mathcal{C}(z^-)\cup\mathcal{C}(z^+)).
\]
As a consequence of Lemma \ref{stability_general}, if we know that ${v}$ belongs to a bounded subset of $X$ and that $\mathcal{K}({v})$ is small, then ${v}$ is  $H^1$-close to its $L^2$-projection onto $\mathcal{C}(z^-)\cup\mathcal{C}(z^+)$. In particular, by the first implication of the lemma, we also know that $M({v})$ has a single point.
\end{remark}
\begin{proof}
\textsc{$\bullet$ Step 1. Uniqueness of $m$.} This is a consequence of \cite[Lemma 2.1.]{Schatzman:2002}. We give an alternate proof in our particular case. Let us pick a point $m_0$ in $M({v})$. By definition, $m_0$ minimizes $F(m):=\|{v}-z(\cdot-m)\|^2_{L^2}$ (we will write $F({v},m)$ in case we need to distinguish the dependence on ${v}$; $F'$ and $F''$ will denote anyway the derivatives w.r.t. $m$), where $z:=z^-$ if $d_{L^2}({v},\mathcal{C}(z^-))\leq d_{L^2}({v},\mathcal{C}(z^+))$ and $z=z^+$ otherwise. Compute the first and second derivatives of $F$:
\[
\begin{cases}
F'(m)=2(z'(\cdot-m)\,,\, ({v}-z(\cdot-m)))_{L^2}\ ,\\
F''(m)=2\|z'\|_{L^2}^2-2(z''(\cdot-m)\,,\,{v}-z(\cdot-m))_{L^2}\ .
\end{cases}
\]
In particular, by optimality, one has 
\begin{equation}
\label{optim}
F'(m_0)=2(z'(\cdot-m_0)\,,\, ({v}-z(\cdot-m_0)))_{L^2}=0.
\end{equation}
Let us set $\lambda:=\|z''\|_{L^2}\|z'\|_{L^2}^{-1}>0$. The Cauchy-Schwarz and Young inequalities yield
\begin{equation}\label{convexest}
|F'(m)|\leq 2\|z'\|_{L^2}\sqrt{F(m)}\quad\text{and}\quad F''(m)\geq 2\|z'\|_{L^2}^2-\lambda^{-2}\|z''\|^2_{L^2}-\lambda^{2}F(m)=\|z'\|_{L^2}^2-\lambda^2F(m).
\end{equation}
We now prove that, for $m$ close to $m_0$, $F(m)$ is small so that $F$ is strictly convex. First, by construction,
\[
F(m_0)=\inf F=d_{L^2}({v},\mathcal{C}(z^-)\cup\mathcal{C}(z^+))\leq\delta.
\]
Now, given $\theta>0$, we estimate $F^\ast:=\sup\{F(m)\;:\; |m-m_0|\leq\theta\}$ using the Mean Value theorem and \eqref{convexest}:
\[
F^\ast\leq F(m_0)+\theta\sup\{|F'(m)|\;:\;|m-m_0|\leq \theta\}\leq \delta+2\theta\|z'\|_{L^2}\sqrt{F^\ast}\leq  \delta+2\theta^2\|z'\|_{L^2}^2+\frac12 F^\ast,
\]
where the last inequality is just an application of a Young inequality. We deduce, $F^\ast\leq 4\theta^2\|z'\|_{L^2}^2+2\delta$. If $\theta$ and $\delta$ are small enough (where $\theta$ and $\delta$ depend on $z\in\{z^-,z^+\}$ only), this implies that $F^\ast<\lambda^{-2}\|z'\|_{L^2}^2/2$. In particular, one has $F''(m)\geq \|z'\|_{L^2}^2/2$ whenever $|m-m_0|\leq\theta$. Thus $F$ is strictly convex around $m_0$ so that $m_0$ is an isolated minimizer of $F$. Since $M({v})$ is bounded, it has a finite number of points and the minimal distance between two of those points is greater than $\theta$. Now, if $m_1,m_2$ are two distinct points in $M({v})$, one has
\[
\|z(\cdot-m_1)-z(\cdot-m_2)\|_{L^2}\leq \|z(\cdot-m_1)-{v}\|_{L^2}+\|{v}-z(\cdot-m_2)\|_{L^2}=\sqrt{F(m_1)}+\sqrt{F(m_2)}\leq 2\sqrt{\delta}.
\]
Up to reducing $\delta$ again if needed (with a bound depending on $z^\pm$ only), we get a contradiction because of the estimate
\[
\|z(\cdot-m_1)-z(\cdot-m_2)\|_{L^2}\geq \inf\{\|z-z(\cdot-m)\|_{L^2}\;:\; |m|\geq\theta\}>0.
\]
Indeed, the existence of a minimizer $m$ in the last infimum follows from both semicontinuity and coercivity of $\|z-z(\cdot-m)\|_{L^2}$, and it is clear that $\|z-z(\cdot-m)\|_{L^2}$ cannot vanish since $z$ is injective, by optimality ($z$ minimizes $\mathfrak{E}_{W}$).
\medskip

\textsc{$\bullet$ Step 2. Continuity of the map ${v}\mapsto m({v})$.} We prove that the map ${v}\mapsto m({v})$ is continuous for the $L^2$-topology on the set of functions ${v}\in X$ such that $d_X({v},\Sigma)\le \delta$. Let $({v}_n)_n\subset X$ be a sequence and ${v}_0\in X$ such that $\|{v}_n-{v}_0\|_{L^2}\to 0$, $d_X({v}_0,\Sigma)\le \delta$ and for each $n$, $d_X({v}_n,\Sigma)\le \delta$. Since $({v}_n)_n$ is $L^2$-bounded, from \eqref{uniformBound}, we learn that $m({v}_n)$ is bounded. But, by semicontinuity of the $L^2$-norm, any converging subsequence of $(m({v}_n))_n$ converges to a minimizer of the problem $\min_m \|{v}_0-z^\pm (\cdot-m)\|_{L^2}$. Thus, by uniqueness of the optimal value $m({v}_0)\in M({v}_0)$, we deduce that $(m({v}_n))_n$ converges to $m({v}_0)$.
\medskip

\textsc{$\bullet$ Step 3. Lipschitz behavior of the map ${v}\mapsto m({v})$.} It is enough to find a constant $C$ such that, for ${v}_0,{v}\in X$ with $d_X({v}_0,\Sigma)\le\delta$, $d_X({v},\Sigma)\le\delta$ and for $\|{v}-{v}_0\|_{L^2}$ sufficiently small (with a constant that may depend on ${v}_0$), we have $|m({v})-m({v}_0)|\leq C\|{v}-{v}_0\|_{L^2}$ (i.e., the Lipschitz behavior may be proven locally). We shall see that we can take $C= 4/\|z'\|_{L^2} $, where $z$ is chosen between $z^-$ and $z^+$ as the $L^2$-closest to ${v}_0$ (note that the functions $z$ associated in this way to ${v}$ and ${v}_0$ are the same if $\delta$ is small). By continuity, we infer that, for $\|{v}-{v}_0\|_{L^2}$ small enough, we have $|m({v})-m({v}_0)|<\theta$ (where $\theta$ is the value used above), so that $F''({v}_0,m)$ is bounded from below by $\|z'\|_{L^2}^2/2$ for $m\in [m({v}),m({v}_0)]$. Then we use
$$\frac{\|z'\|_{L^2}^2}{2}|m({v})-m({v}_0)|\leq |F'({v}_0,m({v}))-F'({v}_0,m({v}_0))|= |F'({v}_0,m({v}))-F'({v},m({v}))|\leq 2\|z'\|_{L^2}\|{v}-{v}_0\|_{L^2},$$
where the first inequality comes from the lower bound on $F''({v}_0,m)$, the next equality comes from $0=F'({v},m({v}))=F'({v}_0,m({v}_0))$ and the last inequality from the formula for $F'$. This implies $|m({v})-m({v}_0)|\leq C\|{v}-{v}_0\|_{L^2}$ with $C=\frac{4}{\|z'\|_{L^2}}$, and proves the claim.

\medskip
\textsc{$\bullet$ Step 4. $L^\infty$ estimate in the second implication: proof of $\mathrm{(2.a.)}$.} We use the same steps that in the proof of Lemma \ref{stability}: we take a sequence $({v}_n)\subset X$ with $\mathcal{K}({v}_n)\to 0$, $d_X({v}_n,\Sigma)\leq R$, and we try to prove that $\|{v}_n-z_n\|_{L^\infty}$ tends to $0$ as $n\to\infty$, $z_n$ being an $L^2$-projection of ${v}_n$ onto $\mathcal{C}(z^-)\cup\mathcal{C}(z^+)$. Without loss of generality, one can assume that
\[
\|{v}_n-z_n\|_{L^\infty}\underset{n\to\infty}{\longrightarrow}\limsup\limits_{n\to\infty}\|{v}_n-z_n\|_{L^\infty},
\]
so that we are free to extract a subsequence whenever needed. By definition, one has $z_n=z^\pm(\cdot-m_n)$ with $m_n\in M({v}_n)$. Up to replacing ${v}_n$ and $z_n$ by ${v}_n(\cdot+m_n)$ and $z_n(\cdot+m_n)$ respectively, one can assume that $m_n=0$ and, for the sake of simplicity, we also assume that $z_n=z^+$. Then, thanks to \eqref{diameter}, we know that the sets $M({v}_n)$ are all included in a fixed compact set:
\[
\exists S_1>0,\, \forall n\geq 0,\, M({v}_n)\subset [-S_1,S_1].
\]
The rest of the proof is similar to that of Lemma \ref{stability}. There are only few little changes. First of all, Claim \ref{centered} has to be replaced by
\begin{claim}\label{stab}
There exist $\delta_0,\varepsilon_1>0$ and $S_0\geq 0$ depending on $W$, $S_1$, $a^-$ and $a^+$ only, such that for all ${v}\in X$ with $\mathcal{K}({v})\leq\delta_0$, $d_X({v},\Sigma)\leq R$, $M({v})\subset [-S_1,S_1]$ and $s\geq S_0$, one has
\[
|{v}(s)-a^-|>\varepsilon_1\quad\text{and}\quad |{v}(-s)-a^+|>\varepsilon_1.
\]
\end{claim}
\noindent Once Claim \ref{stab} is proved (it will be done in a while), since Claim \ref{claimOtherZero} is still valid in the unsymmetric case, the same proof as that of Lemma \ref{stability} shows that $({v}_n)_n$ converges uniformly to some zero $z\in\mathcal{C}(z^-)\cup\mathcal{C}(z^+)$ of $\mathcal{K}$. Note that without symmetry condition, we cannot assert $z=z^\pm$ (as we did in \eqref{dingdong}) but only $z=z^\pm(\cdot-m)$ with $m\in\R$. In order to prove that ${v}_n\to z^+$ uniformly, we need to prove $z=z^+$. Given an interval $I=[-S,S]$ with $S>0$, since $z_n=z^+$ is the projection of ${v}_n$ onto $\mathcal{C}(z^-)\cup\mathcal{C}(z^+)$, one has
\begin{eqnarray*}
\|{v}_n-z^+\|_{L^2(\R)}^2&\leq& \|{v}_n-z\|_{L^2(\R)}^2\leq \|{v}_n-z\|_{L^2(I)}^2+\big(\|{v}_n-z^+\|_{L^2(I^c)}+\|z^+-z\|_{L^2(I^c)}\big)^2\\
&\leq& \|{v}_n-z\|_{L^2(I)}^2+\|{v}_n-z^+\|_{L^2(I^c)}^2+\|z^+-z\|_{L^2(I^c)}^2+2\|z^+-z\|_{L^2(I^c)}\|{v}_n-z^+\|_{L^2(I^c)}.
\end{eqnarray*}
Substracting from both sides $\|{v}_n-z^+\|_{L^2(I^c)}^2$, and using $\|{v}_n-z^+\|_{L^2(\R)}=d_X({v}_n,\Sigma)\leq R$ and the fact that $\|z^+-z\|_{L^2(\R)}$ is independent of $n$ (it only depends on $m$), for all $S>1$ one gets
\[
\|{v}_n-z^+\|_{L^2([-S,S])}^2\leq  \|{v}_n-z\|_{L^2([-S,S])}^2+2R\|z^+-z\|_{L^2(\R\setminus [-S,S])}.
\]
We now pass to the limit as $n\to\infty$ and use the uniform convergence of ${v}_n$ to $z$, we obtain
\[
\|z-z^+\|_{L^2([-S,S])}^2\leq  2R\|z-z^+\|_{L^2(\R\setminus [-S,S])}.
\]
Now, this inequality, for fixed functions $z,z^+$ with $z-z^+\in L^2(\R,\R^n)$, cannot be true for large $S$, unless $z=z^+$. 
\medskip

It remains to prove Claim \ref{stab}. Assume that the conclusion of the claim fails: there exists ${v}\in X$ such that $\mathcal{K}({v})\le\delta_0$, $M({v})\subset [-S_1,S_1]$ and $s_0\geq S_0$, but (for instance) $|{v}(s_0)-a^-|\leq \varepsilon_1$. In particular, one has
\[
\mathfrak{E}_{W}({v})= \mathfrak{E}_{W}({v},(-\infty,s_0])+\mathfrak{E}_{W}({v},[s_0,+\infty))\geq \mathfrak{E}_{W}({v},(-\infty,s_0])+d_K(B(a^-,\varepsilon_1),a^+).
\]
Since $d_K$ is locally equivalent to the Euclidean distance, we know that $d_K(B(a^-,\varepsilon_1),a^+)$ tends to $d_K(a^-,a^+)$ as $\varepsilon_1\to 0$. Moreover, as $\mathcal{K}({v})^2=\mathfrak{E}_{W}({v})-d_K(a^-,a^+)<\delta_0^2$, we have proved that
\[
\mathfrak{E}_{W}({v},(-\infty,s_0]) \underset{\delta_0,\varepsilon_1\to 0}{\longrightarrow} 0.
\]
Since for all $s\leq s_0$, one has $d_K(a^-,{v}(s))\leq \mathfrak{E}_{W}({v},(-\infty,s_0])$, we know that ${v}(s)$ stays in a $d_K$-neighborhood (and so also an Euclidean-neighborhood) of $a^-$: let us say ${v}((-\infty,s_0])\subset B(a^-,r)$ with $r=r(\delta_0,\varepsilon_1)\to 0$ as $\delta_0,\varepsilon_1\to 0$. Now, for all $m\in M({v})\subset [-S_1,S_1]$, one has
\begin{equation*}
R^2\geq d_{L^2}({v},\mathcal{C}(z^-)\cup\mathcal{C}(z^+))^2=\|{v}-z^\pm(\cdot-m)\|_{L^2}^2\geq\int_{s_0/2}^{s_0}\mathrm{dist}(z^\pm(s-m),B(a^-,r))^2\d s.
\end{equation*}
Since $m$ lies on the bounded set $[-S_1,S_1]$, for small $r$ the quantity $\mathrm{dist}(z^\pm(s-m),B(a^-,r))^2$ is bounded from below by a positive constant on the interval $[S_0,+\infty)$. Hence the last term tends to $+\infty$ as $s_0\to \infty$. Thus there is a contradiction if $\delta_0,\varepsilon_1$ are chosen small enough and if $S_0$ is chosen large enough.

\medskip
\textsc{$\bullet$ Step 5. $H^1$ estimate in the second implication: proof of $\mathrm{(2.b.)}$.}
We apply a Taylor-Lagrange expansion to $W$ between the two points ${v}(s)$ and $z(s)$, with $s\in\R$: there exists $\xi(s)\in [{v}(s),z(s)]$ such that the following holds (note that there is no linear part since $z$ minimizes $\mathfrak{E}_{W}$),
\[
\mathcal{K}({v})^2=\mathfrak{E}_{W}({v})-\mathfrak{E}_{W}(z)=\frac 12\int_\R |{v}'(s)-z'(s)|^2\d s+\frac 12\int_\R D^2W(\xi(s))({v}(s)-z(s),{v}(s)-z(s))\d s.
\]
Applying the previous step to small values of $\varepsilon=\varepsilon_0$, one gets $\|{v}-z\|_{L^\infty}\leq \varepsilon_0$. In particular, ${v}$ and $\xi$ are bounded. Moreover, as $W$ is $\mathcal{C}^2$, one has $|\nabla^2W(z(s))-\nabla^2W(\xi(s))|\leq\eta$ for all $s$, whatever $\eta>0$ (up to chosing $\varepsilon_0$ small enough). In particular, $\nabla^2W(\xi(\cdot))$ is bounded and we get the estimate
\[
\frac 12\|{v}'-z'\|_{L^2}^2\leq \mathcal{K}^2({v})+C_0\|{v}-z\|_{L^2}^2.
\]
Thus it remains to prove that $\|{v}-z\|_{L^2}^2\leq C\mathcal{K}^2({v})$. First observe that
\[
\mathcal{K}^2({v})=\frac 12(A(z)({v}-z),({v}-z))_{L^2}+\frac 12\int_\R (D^2W(\xi)-D^2W(z))({v}-z,v-z)\d s,
\]
where $A(z)$ has been defined in \eqref{definitionSpectralA}. The last integral is controlled by $\eta \|{v}-z\|_{L^2}^2$, thus it is enough to prove that $(A(z)({v}-z),({v}-z))_{L^2}$ is larger that $c_0 \|{v}-z\|_{L^2}^2$ with $c_0>0$ (and then choose a small value of $\eta$ so that $\eta<c_0$) ; but this is a consequence of the spectral assumption $\mathbf{(Spec)}$ or, more precisely, \eqref{maxMin} since $z'$ is orthogonal to ${v}-z$ by \eqref{optim}.
\end{proof}
The following lemma claims that if $\gamma(t,\cdot)$ is an absolute continuous curve lying in a neighborhood of $\mathcal{C}(z^-)\cup\mathcal{C}(z^+)$, then the unique element $m(t)$ of $M(\gamma(t,\cdot))$ defines an absolutely continuous function of $t$, and provides an estimate on its variation. This will be useful to study the behavior of $m(t)$ as $t\to\pm\infty$, and thus the existence of the translation parameters $c^-,c^+$ in the system \eqref{Double_connections}, if $\gamma$ minimizes $\mathfrak{L}_{\mathcal{K}}$.
\begin{lemma}\label{translationspeed}
There exists $\delta>0$ (assumed to be small enough for the first implication in Lemma \ref{stability_general} to hold) such that the following property holds true. Let $\gamma\in AC_{loc}(I,X)$ be an injective curve parametrized by $\gamma(t)=(s\mapsto\gamma(t,s))$ such that $\mathfrak{L}_{\mathcal{K}}(\gamma)<+\infty$, $d_X(\gamma(t),\Sigma)\leq\delta$, and for a.e. $t\in I$, $(\partial_t\gamma,\partial_s\gamma)_{L^2}=0$. Then the function $m:I\to \R$ defined by $m(t):=m(\gamma(t))$ is absolutely continuous and there exists $C>0$ (independent of $t$) such that
\[
|m'(t)|\leq C|\dot{\gamma}|(t)\mathcal{K}(\gamma(t)).
\]
\end{lemma}
\begin{proof}
We first set $z:=z^-$ if $d_{L^2}(\gamma(t),\mathcal{C}(z^-))\leq d_{L^2}(\gamma(t),\mathcal{C}(z^+))$ for all $t$ and $z:=z^+$ otherwise: note that $z$ is well defined in this way if $\delta$ is small enough since the $L^2$-distance between $\mathcal{C}(z^-)$ and $\mathcal{C}(z^+)$ is positive and $d_{L^2}(\gamma(t),\mathcal{C}(z^-)\cup\mathcal{C}(z^+))\le\delta$. We remind that $m(t)$ satisfies the optimality condition \eqref{optim}, which rewrites
\begin{equation}
\label{implicit}
0=G(t,m(t)):=\int_\R z'(s)\cdot (\gamma(t,s+m(t))-z(s))\d s.
\end{equation}
We claim that this is a characterization of $m(t)$ for $t$ close to any point $t_0\in I$. More precisely, there exists $\sigma>0$ such that for all $t\in (t_0-\sigma,t_0+\sigma)$, $m(t)$ is the only $m\in \R$ such that
\[
|m-m(t_0)|<\theta/2\quad\text{and}\quad G(t,m)=0,
\]
where $\theta=\theta(\delta,z^\pm)>0$ was introduced at the first step in the proof of Lemma \ref{stability_general}. Indeed, since the map $t\mapsto\gamma(t)$ is absolute continuous for the $L^2$-distance by Lemma \ref{lemmaPropX}, if $|t-t_0|<\sigma$ with $\sigma$ small enough, one has $d_{L^2}(\gamma(t),\gamma(t_0))<\delta$ and so
\[
\|z(\cdot-m(t))-z(\cdot-m(t_0))\|_{L^2}\leq d_{L^2}(\gamma(t),\gamma(t_0))+d_{L^2}(\gamma(t),\mathcal{C}(z))+d_{L^2}(\gamma(t_0),\mathcal{C}(z))\leq 3\delta.
\]
As before, this implies that $|m(t)-m(t_0)|<\theta/2$ (up to reducing $\delta$ if needed). Since $m\mapsto G(t,m)$ is strictly convex on the interval $(m(t_0)-\theta/2,m(t_0)+\theta/2)\subset(m(t)-\theta,m(t)+\theta)$, it is then clear that $m(t)$ is the only zero of $G(t,\cdot)$ on this interval. Now, by the implicit function theorem, we deduce that $t\to m(t)$ is absolutely continuous, and differentiating with respect to $t$ yields
\[
0=\int_\R z'(s)\cdot [\partial_t\gamma(t,s+m(t))+m'(t)\,\partial_s\gamma(t,s+m(t))]\d s=\int_\R z'(s-m(t))\cdot [\partial_t\gamma(t,s)+m'(t)\,\partial_s\gamma(t,s)]\d s.
\]
Equivalently, we have shown the identity
\begin{align*}
0=\int_\R (z'(\cdot-m(t))-\partial_s\gamma)\cdot (\partial_t\gamma+m'(t)\,\partial_s\gamma)\d s+\int_\R\partial_s\gamma\cdot (\partial_t\gamma+m'(t)\,\partial_s\gamma)\d s.
\end{align*}
Since $(\partial_t\gamma,\partial_s\gamma)_{L^2}=0$, we have the estimate
\[
|m'(t)|\,\|\partial_s\gamma(t,\cdot)\|^2_{L^2}\leq \|z'(\cdot-m(t))-\partial_s\gamma(t,\cdot)\|_{L^2}\left[\|\partial_t\gamma(t,\cdot)\|_{L^2}+|m'(t)|\, \|\partial_s\gamma(t,\cdot)\|_{L^2}\right].
\]
Now, for those points $t\in I$ for which $\mathcal{K}(\gamma(t))<\delta$, the $H^1$ estimate in Lemma \ref{stability_general} yields a positive constant $C$ ($=\alpha^{-1}$) such that $\|z'(\cdot-m(t))-\partial_s\gamma(t,\cdot)\|_{L^2}\leq C\mathcal{K}(\gamma(t))$. Thus
\begin{equation}
\label{estbelow}
|m'(t)|\, \|\partial_s\gamma(t,\cdot)\|_{L^2}\left[\|\partial_s\gamma(t,\cdot)\|_{L^2}-C\mathcal{K}(\gamma(t))\right]\leq C\mathcal{K}(\gamma(t))\, \|\partial_t\gamma(t,\cdot)\|_{L^2}.
\end{equation}
We need to prove that $\|\partial_s\gamma\|_{L^2}$ is bounded fromt below. We use
\[
\|\partial_s\gamma(t,\cdot)\|_{L^2}\geq \|z'(\cdot-m(t))\|_{L^2}-\|z'(\cdot-m(t))-\partial_s\gamma(t,\cdot)\|_{L^2}\geq \|z'\|_{L^2}-C\mathcal{K}(\gamma(t)).
\]
When $\mathcal{K}(\gamma(t))$ is small, this implies that $\|\partial_s\gamma(t,\cdot)\|_{L^2}$ is bounded from below. Thus, from \eqref{estbelow}, we deduce the existence of two constants $\delta_0>0$ and $C_0>0$ such that for those points $t$ for which $\mathcal{K}(\gamma(t))<\delta_0$, we have
\[
|m'(t)|\leq C_0 \mathcal{K}(\gamma(t))\, |\dot{\gamma}|(t).
\]
It remains to treat the case $\mathcal{K}(\gamma(t))\geq\delta_0$, but in this case it is enough to use the Lipschitz bounds on the projection proven in Lemma \ref{stability_general}.
\end{proof}
From the previous lemma, we also deduce $L^\infty$ bounds on $m(t)$ (with no assumption on the distance to $\mathcal{C}(z^-)\cup\mathcal{C}(z^+)$):
\begin{lemma}\label{translationbound}
Let $\gamma\in AC_{loc} ([0,1],X)$ be an injective curve, parametrized by $\gamma(t)=(s\mapsto\gamma(t,s))$. Assume that $\mathfrak{L}_{\mathcal{K}}(\gamma)<+\infty$, $\gamma(0)=\mathcal{C}(z^-)$, $\gamma(1)=\mathcal{C}(z^+)$, and $(\partial_t\gamma(t,\cdot),\partial_s\gamma(t,\cdot))_{L^2}=0$ for a.e. $t\in [0,1]$. Then there exists a bounded function $m:[0,1]\to \R$ such that for all $t\in [0,1]$, $m(t)\in M(\gamma(t))$.
\end{lemma}
\begin{proof}
Since $\gamma$ is injective, one has $\mathcal{K}(\gamma(t))>0$ for all $t\in (0,1)$. By lower semicontinuity, one has also $ \mathcal{K}_J:=\inf_{t\in J} \mathcal{K}(\gamma(t))>0$ for every compact interval $J\subset (0,1)$. In particular, as $\mathfrak{L}_{K}(\gamma_{|J})\geq  \mathcal{K}_J L_1 (\gamma_{|J})$, $\gamma$ is bounded for the $L^2(\R)$ distance over $J$. Once again, the following estimate,
\[
\|z^\pm-z^\pm(\cdot-m(t))\|_{L^2}\leq d_{L^2}(z^\pm,\gamma(t))+d_{L^2}(\gamma(t),z^\pm(\cdot-m(t)))\leq 2d_{L^2}(\gamma(t),z^\pm)\leq C,
\]
provides a bound on $m(t)$ for $t\in J$. Since $\gamma$ tends to $\mathcal{C}(z^\pm)$ on the boundary of $[0,1]$, it is clear that $m(\cdot)$ is actually bounded up to the boundary thanks to Lemma \ref{translationspeed}, which provides a bound on the total variation of $m(\cdot)$ for $\gamma(t)$ close to $\mathcal{C}(z^-)\cup\mathcal{C}(z^+)$.
\end{proof}
As in the previous part, the last ingredient in the proof of Theorem \ref{Schatzman_thm} is the following lemma, which is the twin brother of Lemma \ref{s0}:
\begin{lemma}\label{s0twin}
Let $\gamma\in AC_{loc} ([0,1],X)$ be an injective curve, parametrized by $\gamma(t)=(s\mapsto\gamma(t,s))$. Assume that $\mathfrak{L}_{\mathcal{K}}(\gamma)<\infty$, $\gamma(0)=\mathcal{C}(z^-)$, $\gamma(1)=\mathcal{C}(z^+)$, $(\partial_t\gamma(t,\cdot),\partial_s\gamma(t,\cdot))_{L^2}=0$ for a.e. $t\in [0,1]$, and assume that $(t,s)\mapsto\gamma(t,s)$ is continuous. Then, for all $\varepsilon_0>0$, there exist $s^-,s^+\in \R$ such that $s^-< s^+$ and for almost all $t\in [0,1]$, $|\gamma(t,s^\pm)-a^\pm|<\varepsilon_0$.
\end{lemma}
\begin{proof}
The proof is similar to that of Lemma \ref{s0}. We recall the proof of the existence of $s^+$ (the proof of the existence of $s^-$ works the same). First, for all $S>0$, we have the estimate
\begin{equation}\label{meanLK}
\frac 1S\int_{S}^{2S}\int_0^1 \mathcal{K}(\gamma(t))|\partial_t\gamma(t,s)|\d t\d s\leq S^{-1/2}\mathfrak{L}_{\mathcal{K}}(\gamma).
\end{equation}
Since $\mathfrak{L}_{\mathcal{K}}(\gamma)$ is finite, the curve $t\mapsto \gamma(t)$ is bounded in $X$: there is a constant $R>0$ with $d_X(\gamma(t),\Sigma)\leq R$. Thus we can apply Lemma \ref{stability_general} to $\varepsilon=\varepsilon_0/3$: one gets a constant $\delta>0$ such that $\mathcal{K}({v})\leq\delta$ implies that $M(v)$ is reduced to a single point and $\|{v}-z\|_{L^\infty}<\varepsilon_0/3$, where $z$ is the $L^2$-projection of $u$ onto $\mathcal{C}(z^-)\cup\mathcal{C}(z^+)$. Let $I\subset (0,1)$ be the (open) set of instants $t$ such that $\mathcal{K}(\gamma(t))>\delta$. For every $t\in [0,1]$, let $z^\pm(\cdot-m(t))$ be the projection of $\gamma(t)$ onto $\mathcal{C}(z^-)\cup\mathcal{C}(z^+)$, where $t\mapsto m(t)\in M(\gamma(t))$ is the bounded function provided by Lemma \ref{translationbound}. Thus, for all $t\in [0,1]\setminus I$, one has
\[
\|\gamma(t,\cdot)-z^\pm(\cdot-m(t))\|_{L^\infty}<\varepsilon_0/3.
\]
Since $t\mapsto m(t)$ is bounded, there exists $S>0$ large enough so that
\[
S^{-1/2}\mathfrak{L}_{\mathcal{K}}(\gamma)<\frac{\delta\varepsilon_0}{3}\quad \text{and}\quad\forall t\in [0,1],\, \forall s>S,\, |z^\pm (s-m(t))-a^+|<\frac{\varepsilon_0}{3}.
\]
Thus, by \eqref{meanLK}, there exists $s^+\in[S,2S]$ such that 
\begin{equation}
\label{good_s}
\int_0^1\mathcal{K}(\gamma(t))|\partial_t\gamma(t,s^+)|\d t< \frac{\delta\varepsilon_0}{3}.
\end{equation}
Now, for every $t\in [0,1]\setminus I$, one has the estimate
\[
|\gamma(t,s^+)-a^+|\leq |\gamma(t,s^+)-z^\pm(s^+-m(t))|+|z^\pm(s^+-m(t))-a^+|< \frac{2\varepsilon_0}{3}.
\]
For the points in $I$, \eqref{good_s} allows to estimate the total variation of $t\mapsto\gamma(t,s^+)$ on $I$ as follows,
\[
\int_I|\partial_t\gamma(t,s^+)|\d t\leq\frac{\varepsilon_0}{3}.
\]
Since $t\mapsto \gamma(t,s^+)$ is continuous and $|\gamma(t,s^+)-a^+|\leq \frac{2\varepsilon_0}{3}$ on the boundary of $I$, we have $|\gamma(t,s^+)-a^+|<\varepsilon_0$ which is what had to be proved.
\end{proof}
We are now able to prove the main result of the section:
\begin{proof}[Proof of Theorem \ref{Schatzman_thm}]
\textsc{$\bullet$ Step 1:  Existence of an $\mathfrak{E}_{\mathcal{W}}$-minimizing curve on $X$ between $\mathcal{C}(z^-)$ and $\mathcal{C}(z^+)$.} The proof of this first step is rigorously the same that the first step in the proof of Theorem \ref{ABGthm} except the fact that we use Lemma \ref{translation} and Lemma \ref{s0twin} instead of Lemma \ref{s0}, and that we use an avatar of Lemma \ref{smoothing} adapted to this non-symmetric context (the regularization process does not use the symmetry condition). Thus one gets the existence of a curve $\gamma\in\Lip(\R,X)$ such that
\[
\forall\sigma\in AC_{ploc}(\R,X),\, \sigma:z^-\mapsto z^+\Longrightarrow \mathfrak{E}_{\mathcal{W}}(\gamma)\leq \mathfrak{E}_{\mathcal{W}}(\sigma),
\]
where $\mathcal{W}=\mathcal{K}^2/2$. Moreover, the proof insures that such a curve lies on a funnel with exponential or polynomial decay: for all $t\in [0,1]$, one has
\begin{equation}
\label{funnelbnd}
\gamma(t)\in\mathcal{C}:=\{{v}:\R\to\R^n\text{ measurable s.t. for a.e. $s$ with }\pm(s-s^\pm)\geq 0,\, |{v}(s)-a^\pm|\leq E(s)\}\subset X,
\end{equation}
with $s^-<s^+$, and where the rate of convergence at infinity is given by the function $E(\cdot)$ (see Lemma \ref{projection}).

\medskip
\textsc{$\bullet$ Step 2: boundary conditions.} It is clear, by \eqref{energy_length} and by construction of $\gamma$, that ${u}\in L^2(\R^2,\R^n)$, defined by ${u}(x_1,x_2)=\gamma(x_2,x_1)$, minimizes $\mathcal{E}$ under the constraints detailed in Theorem \ref{Schatzman_thm}, i.e. \eqref{BC_Schatzman}. Moreover, by \eqref{funnelbnd}, we know that ${u}(x_1,x_2)$ converges to $a^\pm$ as $x_1\to\pm\infty$, uniformly in $x_2$. In order to prove that $\gamma$ solves \eqref{Double_connections}, we need to prove the existence of two parameters $c^-,c^+\in\R$ such that
\begin{equation}\label{toprove}
\begin{cases}
{u}(x_1,x_2)\to z^-(x_1-c^-)&\text{when }x_2\to -\infty,\text{ uniformly w.r.t. }x_1;\\
{u}(x_1,x_2)\to z^+(x_1-c^+)&\text{when }x_2\to +\infty,\text{ uniformly w.r.t. }x_1.
\end{cases}
\end{equation}
We first prove convergence in $L^2(\R)$ (in the variable $x_1$). Note that, by construction, we already know that for $x_2$ close to $\pm\infty$, ${u}(\cdot,x_2)$ is close to $\mathcal{C}(z^\pm)$ so that there is a unique point $m(x_2)$ in $M({u}(\cdot,x_2))$ (see Lemma \ref{stability_general}). Thus
\begin{equation}\label{firststep}
\lim\limits_{x_2\to\pm\infty}\|{u}(\cdot,x_2)-z^\pm(\cdot-m(x_2))\|_{L^2(\R)}=0.
\end{equation}
Moreover, since by optimality $\gamma$ is injective and for a.e. $t\in\R$, $(\partial_t\gamma\,;\,\partial_s\gamma)=0$ (otherwise, by Lemma \ref{translation} and its proof, one could strictly reduce $\mathfrak{L}_{\mathcal{K}}(\gamma)$ by translating each of the curves $\gamma(t,\cdot)$), we can apply Lemma \ref{translationspeed} which says that the function $x_2\mapsto m(x_2)$ is of bounded variations in a neighborhood of $\pm\infty$, let say for $|x_2|\geq T>0$, and that we have the estimate
\[
\|m'\|_{L^1(\R\setminus [-T,T])}\leq C \mathfrak{L}_{\mathcal{K}}(\gamma)<+\infty.
\]
In particular, $x_2\mapsto m(x_2)$ has a limit when $x_2\to\pm\infty$: there exist $c^\pm$ with 
\[
c^\pm=\lim\limits_{x_2\to\pm\infty}m(x_2).
\]
Together with \eqref{firststep} and with the continuity of the translations in $L^2$, this implies
\[
\lim\limits_{x_2\to\pm\infty}\|{u}(\cdot,x_2)-z^\pm(\cdot-c^\pm)\|_{L^2(\R)}=0.
\]
\textsc{$\bullet$ Step 3: Improvement of the convergence as $x_2\to\pm\infty$ and of the regularity.} As in the proof of Theorem \ref{ABGthm}, we can use the result presented in the Appendix (Corollary \ref{Wbounded}) in order to get uniform convergence as $x_2\to\pm\infty$ and boundedness for ${u}$. In particular, ${u}$ solves the Euler-Lagrange equation $-\Delta {u}+\nabla W({u})=0$ associated to the energy $\mathcal{E}$ (which makes sense since $u$ is bounded) and, by a boot-strap argument, we obtain $C^{2,\alpha}$ and possibly higher regularity.
\end{proof}

\appendix

\section{About the condition {\bf (H3a)}}

Condition {\bf (H3a)} was crucial in \cite{nous} and in finite-dimensional heteroclinic connection problems (see also \cite{FusGroNov} where this condition is cited). It was also introduced in \cite{BraButSan} for applications to weighted distances in Wasserstein spaces. In the present paper, the role of {\bf (H3a)} is less important, as it requires to be coupled with {\bf (H3b)}, which was not satisfied in the examples that we analyzed in Sections \ref{sectionAlama} and \ref{sectionSchatzman}.

However, we think that it is important to discuss this assumption as it seems that it has been neglected for long by specialists of heteroclinic connections (while it was considered natural in other communities). In particular, we want here to provide an example where existence of geodesics fails in a case where {\bf (H3a)} is not satisfied. We will consider an Euclidean space, say $\R^2$, endowed with a positive weight $K\geq 0$, with a finite number of wells, but which does not satisfy {\bf (H3a)}. 

\begin{theorem}
Given an arbitrary continuous and strictly positive function $g:[1,+\infty)\to (0,\infty)$ such that $\int_1^\infty g(s)ds<+\infty$ there exists a weight $K:\R^2\to\R_+$ such that
\begin{itemize}
\item $\Sigma=\{K=0\}$ is finite;
\item the weight $K$ coincides with $g(d(\cdot,\Sigma))$ on an unbounded set;
\item there exist two points $P_\pm\in\Sigma$ such that there is no curve connecting $P_+$ to $P_-$ minimizing the $K$-length.
\end{itemize}
\end{theorem}
\begin{proof}
First we extend the function $g$ of the statement to the whole $\R_+$, defining it on $[0,1]$ in such a way that $g(0)=0$ and $g>0$ on $(0,1)$ (take for instance $g(s)=sg(1)$). Then define $G(t)=\int_0^t g(s)\d s$ and $G(\infty):=\int_0^\infty g(s)\d s\in\R$. Also take a smooth function $h:\R\to [0,1]$ such that $h(0)=1$ and $h(\pm 1)=0$ (which implies $h'(0)=h'(\pm 1)=0$), and choose it so that $h'$ only vanishes at $-1,0,1$. Build a function $f:\R^2\to\R_+$ via 
$$f(x,y)=h(y)(2G(\infty)-G(|x|))+(1-h(y))G(|x|).$$
Note that we have $f\in C^1(\R^2)$ (indeed, $x\mapsto G(|x|)$ is $C^1$ despite the absolute value, because $g(0)=0$). Then take $K=|\nabla f|$. This means that we are defining
$$K(x,y)=\sqrt{|1-2h(y)|^2g(|x|)^2+4h'(y)^2(G(\infty)-G(|x|))^2}.$$
It is easy to see that $K(x,0)=g(|x|)$.

Look at $\Sigma=\{K=0\}$. From the strict inequality $G(\infty)-G(|x|)>0$, in order to have $K=0$ we need $h'(y)=0$, but this implies $y=-1,0,1$. For these values of $y$, we have $h(y)\neq 1/2$: thus, for $K$ to vanish, we also need $g(|x|)=0$, i.e. $x=0$, since $g$ was supposed to be strictly positive elsewhere. Then $\Sigma=\{(0,-1),(0,0), (0,1)\}$. In particular, on the line $\{y=0\}$, the distance to $\Sigma$ coincides with the distance to the origin, i.e. with $|x|$, and we have $K=g(d(\cdot,\Sigma))$ on such a line (considering that $g$ was originally only defined on $[1,\infty)$, the equality between $K$ and $g(d(\cdot,\Sigma))$ does not hold on the whole line, but on an unbounded part of it).

Consider now the two points $P_+:=(0,1)$ and $P_-:=(0,-1)$, which belong to $\Sigma$. 

Any curve $\gamma:[0,1]\to \R^2$ connecting $P_+$ and $P_-$ must cross the line $\{y=0\}$. Let us call $(x_0,0)$ a point where $\gamma$ crosses such a line, i.e. $\gamma(t_0)=(x_0,0)$ for $t_0\in [0,1]$. We can estimate the weighted length of $\gamma$ via
\begin{eqnarray*}
\int_0^1 K(\gamma(t))|\gamma'(t)|\d t
							&=&\int_0^{t_0} |\nabla f|(\gamma(t))|\gamma'(t)|\d t+\int_{t_0}^1 |\nabla f|(\gamma(t))|\gamma'(t)|\d t\\
						&\geq&\int_0^{t_0} \nabla f(\gamma(t))\cdot \gamma'(t)\d t-\int_{t_0}^1 \nabla f(\gamma(t))\cdot\gamma'(t)\d t\\
&=&2f(\gamma(t_0))-f(P_+)-f(P_-)=2f(x_0,0)-f(0,1)-f(0,-1)\\
&=&2(2G(\infty)-G(|x_0|))>2G(\infty).
\end{eqnarray*}
This shows that no curve joining these two points can have a weighted length less or equal than $2G(\infty)$. If we are able to construct a sequence of curves approaching this value, we have shown that the minimum of the weighted length does not exist. To do so, take a sequence $x_n\to\infty$ with $g(x_n)\to 0$ (which is possible because of the condition $\int_0^\infty g(s)\d s<+\infty$). Consider a curve $\gamma_n$ defined on the interval $[0,2x_n+2]$ (and possibly reparametrized on $[0,1]$) in the following way
$$\gamma_n(t)=\begin{cases} (t,1)&\mbox{ if }t\in [0,x_n],\\
						(x_n,1-t+x_n) &\mbox{ if }t\in [x_n,x_n+2],\\
						(2x_n+2-t,-1)&\mbox{ if }t\in [x_n+2,2x_n+2].\end{cases}$$

It is easy to see that 
$$\int_0^{2x_n+2} K(\gamma(t))|\gamma'(t)|\d t=2\int_0^{x_n} g(t)\d t+\int_{-1}^1\sqrt{|1-2h(y)|^2g(|x_n|)^2+4h'(y)^2(G(\infty)-G(|x_n|))^2}\d y.$$
Using
$$\sqrt{|1-2h(y)|^2g(|x_n|)^2+4h'(y)^2(G(\infty)-G(x_n))^2}\leq C(g(|x_n|)+G(\infty)-G(x_n))$$
it is easy to see
$$\int_0^{2x_n+2} K(\gamma(t))|\gamma'(t)|\d t=2G(x_n)+O(g(|x_n|)+G(\infty)-G(x_n))\underset{n\to\infty}{\longrightarrow} 2G(\infty),$$
which concludes the example.
\end{proof}
Note that in the example above the fact that $K$ vanishes on a finite number of points and the choice of connecting two wells is arbitrary, and is only made for the sake of consistency with the rest of the paper. Other examples could easily be built.

\section{Improved estimates for minimal action curves in Hilbert spaces}\label{AppB}

We consider here a general case, where we have a functional $\mathcal W:H\to[0,+\infty]$ defined on a Hilbert space $H$, and we consider curves $\gamma:\R\to H$ associated with an action 
$$E(\gamma):=\int_\R \left(\frac 12  |\gamma'(t)|^2+ \mathcal W(\gamma(t))\right)\d t.$$
We consider a curve $\gamma$ which is a local minimizer of $E$ (in the sense that $E(\gamma+\eta)\geq E(\gamma)$ for every compactly supported perturbation $\eta:\R\to H$, and such that $E(\gamma)<+\infty$. We do not consider the problem of minimizing $E$ among all competitors with fixed boundary conditions at $\pm \infty$ because we want to consider the case where $\gamma$ has no limits at $\pm\infty$. This framework includes that of our sections \ref{sectionAlama} and \ref{sectionSchatzman}, where the Hilbert space is $L^2(\R)$. The curves we provided had a limit at $\pm\infty$, but we prefer to ignore this fact. 

We say that the functional $\mathcal W$ is $\lambda$-convex if it satisfies for every $a,b\in H$ and $s\in (0,1)$ the inequality
$$\W((1-s)a+sb)\leq (1-s)\W(a)+s\W(b)-\frac{\lambda}{2}s(1-s)|a-b|^2.$$
Here $\lambda\in\R$ can be negative (for $\lambda\geq 0$ the function $\W$ would be convex). For smooth functions on the Euclidean space, $\lambda$-convexity of $\mathcal W$ coincides with the lower bound $\nabla^2\mathcal W\geq \lambda I$ (where $I$ is the identity matrix, and the inequality is to be intended in the sense of symmetric matrices, i.e. the difference is positive-semidefinite).

We are interested in the following fact.

\begin{theorem}
Suppose that the functional $\W$ is $\lambda$-convex and that $\gamma$ is a local minimizer with finite energy $E$. Then $\gamma''\in L^2(\R;H)$ and $\int |\gamma''(t)|^2\d t\leq C\int |\gamma'(t)|^2\d t$, for a constant $C$ only depending on $\lambda$.
\end{theorem}
\begin{remark}
The idea behind this result is very easy: from the optimality of $\gamma$ we write the Euler-Lagrange equation $\gamma''=\nabla \W(\gamma)$, we multiply it by $\gamma''$ and obtain
$$|\gamma''(t)|^2=\nabla \W(\gamma)\cdot \gamma''=(\W(\gamma(t))''-\nabla^2\W(\gamma(t))(\gamma'(t),\gamma'(t))\leq (\W(\gamma(t))''+|\lambda||\gamma'(t)|^2.$$
Then we integrate times a cut-off function $\eta(t)$ with $\eta=1$ on $[-M,M]$ and $\eta=0$ out of $[-M-1,M+1]$, and use $\int (W\circ\gamma)''\eta=\int (W\circ\gamma) \eta''$. Letting $M\to \infty$ we obtain the required result. 

Unfortunately, this argument requires some regularity and is not easy to perform when $\W$ is not smooth. One could obtain this for general $\lambda$-convex functionals $\W$ by approximation (an interesting fact is that $\lambda$-convexity guarantees uniqueness of the minimizer of $E$ on sufficiently small intervals, which allow to obtain results on any local minimizer of the limit problem by approximation, if it can be written as a limit of minimizers of smooth problems). Yet, such an approximation procedure could require some compactness in $H$, which is not always available in infinite dimension. This is the reason why we will give a different proof, with a non-optimal constant $C$.
\end{remark}
\begin{proof}
The local optimality of $\gamma$ can be used in the following way: take $t\in \R$ and $h>0$, and replace $\gamma(s)$ by the curve 
$$\tilde\gamma(s):=\begin{cases}\gamma(s)&\mbox{ if }|s-t|>h,\\
				(\frac 12 -\frac{s-t}{2h})\gamma(t-h) +(\frac 12+\frac{s-t}{2h})\gamma(t+h)&\mbox{ if }|s-t|\leq h,\end{cases}$$
which essentially means replacing $\gamma$ with the segment joining $\gamma(t-h)$ to $\gamma(t+h)$ on $[t-h,t+h]$. We first obtain, by Jensen's inequality on the kinetic part of the action,
$$
\int_{t-h}^{t+h}\W(\gamma(s))\d s+\frac{|\gamma(t+h)-\gamma(t)|^2}{2h}+\frac{|\gamma(t)-\gamma(t-h)|^2}{2h}
\leq \int_{t-h}^{t+h}\left(\W(\gamma(s))+\frac 12 |\gamma'(s)|^2\right)\d s,
$$
and then, by optimality of $\gamma$ and by $\lambda$-convexity of $\mathcal{W}$,
\begin{multline*}
\int_{t-h}^{t+h}\W(\gamma(s))\d s+\frac{|\gamma(t+h)-\gamma(t)|^2}{2h}+\frac{|\gamma(t)-\gamma(t-h)|^2}{2h}\leq\int_{t-h}^{t+h}\left(\W(\tilde\gamma(s))+\frac 12 |\tilde\gamma'(s)|^2\right)\d s\\
=\int_0^1\left( \W((1-r)\gamma(t-h)+r\gamma(t+h))+\frac{|\gamma(t+h)-\gamma(t-h)|^2}{8h^2}\right)2h\d r\\
\leq h(\W(\gamma(t-h))+\W(\gamma(t+h)))+\frac{|\lambda|\, h}{6} |\gamma(t+h)-\gamma(t-h)|^2+\frac{|\gamma(t+h)-\gamma(t-h)|^2}{4h}.
\end{multline*}
Using the (parallelogram) identity 
$$\frac{|A|^2+|B|^2}{2h}-\frac{|A+B|^2}{4h}=\frac{|A-B|^2}{4h}$$
applied to the vectors $A=\gamma(t+h)-\gamma(t)$ and $B=\gamma(t)-\gamma(t-h)$, we can re-arrange the terms above and divide by $h^3$, thus obtaining
$$\frac{|\gamma(t+h)-2\gamma(t)+\gamma(t-h)|^2}{4h^4}\leq
 \frac{\W(\gamma(t+h))-2\fint_{t-h}^{t+h}W(\gamma(s))\d s+\W(\gamma(t-h))}{h^2}+C\frac{|\gamma(t+h)-\gamma(t-h)|^2}{h^2}.
$$
We now integrate this over $t\in \R$. Using
$$\int_\R\W(\gamma(t+h))\d t=\int_\R\W(\gamma(t-h))\d t=\int_\R\fint_{t-h}^{t+h}W(\gamma(s))\d s\d t$$
we see that the integral of the first term on the right hand side vanishes. Using
$$\frac{|\gamma(t+h)-\gamma(t-h)|^2}{h^2}=4\left|\fint_{t-h}^{t+h}\gamma'(s)\d s\right|^2\leq 4 \fint_{t-h}^{t+h}|\gamma'(s)|^2\d s,$$
we see that the integral of the second term is smaller than $C\int |\gamma'(t)|^2\d t$. Hence we have got
$$\frac 14 \int_\R\left|\frac{\gamma(t+h)-2\gamma(t)+\gamma(t-h)}{h^2}\right|^2\d t\leq C\int_\R |\gamma'(t)|^2\d t.$$
We then obtain the result by letting $h\to 0$, since we have
$$\frac{\gamma(\cdot+h)-2\gamma(\cdot)+\gamma(\cdot-h)}{h^2}\deb \gamma''.\qedhere$$

\end{proof}
We also have the following consequence
\begin{corollary}\label{Wbounded}
Suppose that the functional $\W$ is $\lambda$-convex and that $\gamma$ is a local minimizer with finite energy of $E$. Then $|\gamma'|$ and $\W\circ\gamma$ are bounded and absolutely continuous functions. \end{corollary}
\begin{proof}
We just proved the $L^2$-integrability of $|\gamma''|$ and, a fortiori, of $(|\gamma'|)'$. This shows that $|\gamma'|$ is locally $H^1$, and we have 
$$\left||\gamma'(t)|^2-|\gamma'(t_0)|^2\right|\leq \int_{[t_0,t]} 2|\gamma'(s)| \left|(|\gamma'|)'(s)\right|\d s\leq 2\|\gamma'\|_{L^2}\|\gamma''\|_{L^2}.$$
As a consequence, $|\gamma'|$ is bounded (we can also choose a sequence of points $t_0\to\infty$ such that $|\gamma'(t_0)|\to 0$ if we want a more explicit estimate). Then, using the equipartition of the energy which is true for local minimizers, we also have the same result for $\W\circ \gamma$. 
\end{proof}

\end{document}